\documentclass[11pt]{article}
\usepackage{makeidx, amsmath,amssymb, amsthm, graphicx, color, verbatim,enumitem, slantsc}
\usepackage{multirow}
\usepackage{float}
\usepackage{subcaption}
\usepackage{enumitem}
\usepackage{mathtools}
\usepackage{algorithm}
\usepackage{algorithmic}
\usepackage{hyperref}

\newtheorem{lemma}{Lemma}
\newtheorem{proposition}{Proposition}

\newtheorem{remark}{Remark}
\newtheorem{theorem}{Theorem}

\begin{document}

\title{On the B-subdifferential of proximal operators of affine-constrained $\ell_1$ regularizer}

\author{Xudong Li\thanks{School of Data Science, Fudan University, Shanghai, P.R. China ({\tt lixudong@fudan.edu.cn}).}, \qquad Meixia Lin\thanks{(Corresponding author) Engineering Systems and Design, Singapore University of Technology and Design ({\tt meixia\_lin@sutd.edu.sg}).}, \qquad 
	Kim-Chuan Toh\thanks{Department of Mathematics and Institute of Operations Research and Analytics, National University of Singapore, Singapore ({\tt mattohkc@nus.edu.sg}).}.}
\maketitle
\date

\begin{abstract}
In this work, we study the affine-constrained $\ell_1$ regularizers, which frequently arise in statistical and machine learning problems across a variety of applications, including microbiome compositional data analysis and sparse subspace clustering. With the aim of developing scalable second-order methods for solving optimization problems involving such regularizers, we analyze the associated proximal mapping and characterize its generalized differentiability, with a focus on its B-subdifferential. The revealed structured sparsity in the B-subdifferential enables us to design efficient algorithms within the proximal point framework. Extensive numerical experiments on real applications, including comparisons with state-of-the-art solvers, further demonstrate the superior performance of our approach. Our findings provide new insights into the sensitivity and stability properties of affine-constrained nonsmooth regularizers, and contribute to  the development of fast second-order methods for a class of structured, constrained sparse learning problems.
\end{abstract}

\section{Introduction}

We consider the function $q_{\mu,c}:\mathbb{R}^n \rightarrow \mathbb{R}$ defined by 
\begin{equation}
	q_{\mu,c}(x) := \|x\|_1 + \delta_{\mu,c}(x), \label{eq: def_q}
\end{equation}
where $\delta_{\mu,c}(x)$ is the indicator function of the affine set $C_{\mu,c} = \{ x\in \mathbb{R}^n \mid \mu^\top x = c \}$, taking the value of $0$ if $x\in C_{\mu,c}$, and $ +\infty$ otherwise. Here, $\mu\in \mathbb{R}^n$ is a fixed non-zero vector and $c\in \mathbb{R}$ is a constant scalar. This function $q_{\mu,c}(\cdot)$ combines the $\ell_1$-norm with an affine constraint, giving rise to the affine-constrained lasso penalty. It naturally appears in optimization problems of the form
\begin{equation}
	\min_{x \in \mathbb{R}^n} \  \left\{ F(x):=f(Ax) + \lambda q_{\mu,c}(x)\right\},
	\label{eq: l1_linear_constraint_obj}
\end{equation}
where $f:\mathbb{R}^m \rightarrow \mathbb{R}$ is a convex loss function, $A\in \mathbb{R}^{m\times n}$ is a data matrix, and $\lambda>0$ is a regularization parameter. Such problems arise in a wide range of applications where one seeks a sparse solution subject to an affine constraint, which often reflects intrinsic structural requirements of the data, such as compositionality, linear relations or conservation laws.

One representative example is microbiome compositional data analysis, where each sample consists of relative abundances that sum to one, imposing structural constraints that require specialized regression methods. A well-established approach is the log-contrast model \cite{aitchison1984log}, where a log transformation is applied to the compositional covariates to enable interpretable linear regression analysis. Specifically, let $y \in \mathbb{R}^m$ be the response vector, and $Z \in \mathbb{R}^{m \times n}$ be the covariate matrix with each row lying in the positive probability simplex. By defining $A = \log Z \in \mathbb{R}^{m \times n}$ elementwise, the log-contrast model takes the form:
\[
b = Ax + \varepsilon, \quad \text{subject to } e^\top x = 0,
\]
where $x\in \mathbb{R}^n$ denotes the regression coefficients, $\varepsilon\sim{\cal N}(0,\sigma^2 I_m)$ is the noise vector, and $e$ is the vector of all ones. 
In high-dimensional settings, several works \cite{lin2014variable,shi2016regression,susin2020variable} have proposed imposing sparsity on regression coefficients through $\ell_1$-regularization to enable variable selection, leading to the constrained nonsmooth problem:
\begin{equation}
	\min_{x\in \mathbb{R}^n} \  \left\{\frac{1}{2} \|b - Ax\|^2 + \lambda \|x\|_1 \ \middle\vert\  e^\top x = 0\right\}, \label{eq: micro_ls}
\end{equation}
which fits into the general formulation \eqref{eq: l1_linear_constraint_obj} with $\mu = e$, $c=0$ and a least squares loss function. This model has been shown to effectively identify relevant microbial features while respecting the compositional nature of the data. Building on this framework, Lu et al. \cite{lu2019generalized} extended the methodology to generalized linear models, including log-contrast logistic regression problem:
\begin{equation}
	\min_{x\in \mathbb{R}^n} \ \left\{  \sum_{i=1}^m \log\left(1 + \exp(-b_i a_i^\top x)\right) + \lambda \|x\|_1 \ \middle\vert\  e^\top x = 0 \right\}, \label{eq: micro_logistic}
\end{equation}
where $a_i^\top$ denotes the $i$-th row of the matrix $A$ and $b_i \in \{-1, 1\}$ are binary responses. This problem fits into the general formulation \eqref{eq: l1_linear_constraint_obj} by setting $\mu = e$, $c=0$, and taking the loss function $f(\cdot)$ as the logistic loss.

Beyond compositional models, the affine-constrained $\ell_1$ regularizer \eqref{eq: def_q} also plays a key role in sparse subspace clustering. This widely used approach in unsupervised learning represents each data point as a sparse linear combination of others, under the assumption that the data lie near a union of low-dimensional affine subspaces. An essential step in many modern frameworks \cite{vidal2009sparse,elhamifar2013sparse,traganitis2017sketched,pourkamali2019large,abdolali2019scalable,pourkamali2020efficient} is to solve an affinely constrained $\ell_1$ regularized least squares problem. Specifically, given a data matrix $A\in \mathbb{R}^{m\times n}$ whose columns are data points in $\mathbb{R}^m$, one seeks a coefficient matrix $X$ that yields sparse representations of all points, leading to the optimization problem:
\begin{equation}
	\min_{X \in \mathbb{R}^{n \times n}}\ \left\{ 
	\frac{1}{2}\|A-AX\|_F^2 +\lambda \|X\|_1 \ \middle \vert  \ {\rm Diag}(X) = 0,\, X^\top e = e
	\right\},\label{eq: sub_cluster_matrix}
\end{equation}
where $\|X\|_1 := \sum_{i=1}^n \sum_{j=1}^n |X_{ij}|$.
This problem decouples into column-wise subproblems of the form:
\begin{equation}
	\min_{x\in \mathbb{R}^n}  \left\{ \frac{1}{2}\| Ax-a\|^2 + \lambda \|x\|_1 \ \middle \vert\ e^\top x = 1 \right\},\label{eq: sub_cluster_vector}
\end{equation}
where $a$ is a fixed column of $A$. Here,  although ${\rm Diag}(X)=0$ implies each column essentially lies in $ \mathbb{R}^{n-1}$, we write $x\in \mathbb{R}^n$ for notational convenience. This subproblem matches the general model \eqref{eq: l1_linear_constraint_obj} with $f(\cdot)$ as the least squares loss, $\mu = e$, and $c=1$.

Numerous algorithms have been proposed to solve problems of the form \eqref{eq: l1_linear_constraint_obj}. Zhou and Lange \cite{zhou2013path} introduced a path-following algorithm for the constrained least squares problem without the $\ell_1$ regularization, where they replaced the constraint by an exact penalty formulation. Later, Lin et al. \cite{lin2014variable} tackled the constrained lasso problem with a least squares loss in the log-contrast setting, using the alternating direction method of multipliers (ADMM), with coordinate descent employed for solving subproblems. Subsequent work by Gaines et al. \cite{gaines2018algorithms} explored methods such as quadratic programming, ADMM, and path-following algorithms, to address the same class of problems. Moving beyond least squares, Lu et al. \cite{lu2019generalized} studied the generalized linear models under affine constraints via an accelerated proximal gradient method, while James et al. \cite{james2020penalized} proposed the Penalized and Constrained optimization method (PaC), a modified coordinate descent scheme for computing solution paths of problem \eqref{eq: l1_linear_constraint_obj} with twice-differentiable loss functions. More recently, Tran et al. \cite{tran2024fast} addressed the equality-constrained lasso problem by first performing variable screening using solutions from unconstrained lasso problems, and then refining the results with a hybrid ADMM and Newton–Raphson method. While these methods provide valuable insights and have been applied successfully in various settings, their computational efficiency and scalability are rather limited, particularly in high-dimensional regimes.
This motivates the development of more scalable approaches tailored to sparse optimization problems of the form \eqref{eq: l1_linear_constraint_obj}.

To this end, we investigate the application of the proximal point algorithm (PPA), which has recently been proven to be an effective tool for solving large-scale nonsmooth optimization. However, the practical use of PPA relies on efficiently solving its subproblems to a sufficient level of precision. Inspired by the work of Li et al. \cite{li2018highly}, we develop an efficient second-order semismooth Newton framework that leverages the ``sparse plus low-rank'' decomposition of the subdifferential of a non-standard proximal mapping. Central to our algorithm is to characterize the generalized differentiability of the proximal mapping associated with the affine-constrained $\ell_1$ regularizer $\lambda q_{\mu,c}$:
\begin{equation}
	{\rm Prox}_{\lambda q_{\mu,c}}(x) = \underset{z\in \mathbb{R}^n}{\arg\min} \left\{  \frac{1}{2}\|z-x\|^2+ \lambda  \|z\|_1 \mid \mu^\top z = c \right\}. \label{eq: def_prox_qx_C}
\end{equation} 
The exact solution to \eqref{eq: def_prox_qx_C} is known only in the special case $\mu=e$ and $c=1$, which can be computed in ${\cal O}(n \log n)$ time via a one-dimensional root-finding procedure \cite[Algorithm 2]{pourkamali2020efficient}. However, the analytical form of its 
B-subdifferential has not been established, even for this special case. We adapt the approach in \cite{pourkamali2020efficient} to develop an explicit method for computing ${\rm Prox}_{\lambda q_{\mu,c}}(\cdot)$ for arbitrary $\mu$ and $c$, and further provide the first complete characterization of its B-subdifferential. These results thereby enable a fast, globally convergent Newton-type algorithm for a broad class of affine-constrained $\ell_1$-regularized problems.

The rest of the paper is organized as follows.
We begin in Section~\ref{sec: prox} with the computation of the proximal mapping ${\rm Prox}_{\lambda q_{\mu,c}}(\cdot)$, 
followed in Section~\ref{sec: Bdiff} by a characterization of its B-subdifferential. 
Based on the established results, Section~\ref{sec: alg} introduces a double-loop algorithm for affine-constrained sparse optimization, 
and Section~\ref{sec: exp} presents numerical experiments on representative application problems, comparing our method with existing solvers. 
Finally, Section~\ref{sec: conclusion} concludes the paper.

\medskip
\noindent \textbf{Notation.} Denote $[n] = \{1,2,\cdots,n\}$. We use 
$\operatorname{sign}(x)$ to denote the sign of $x$, i.e., $\operatorname{sign}(x) := 1$ if $x > 0$, $0$ if $x = 0$, and $-1$ if $x < 0$. We also use $(x)_+ := \max\{x, 0\}$ to denote the positive part of $x$. For an index set $J\subseteq [n]$, we use $|J|$ to denote the cardinality of $J$.  For a given set $D\subseteq \mathbb{R}$, let ${\rm 1}_D(x)$ denote the function that equals $1$ if $x\in D$ and $0$ otherwise. 

\section{Computation of the proximal mapping}\label{sec: prox}
We assume, without loss of generality, that $\mu_i\neq  0$ for all $i\in [n]$. This assumption is justified by the fact that, for any non-zero $\mu\in \mathbb{R}^n$ with $I =\{i\in [n]\mid \mu_i \neq 0\}$, the proximal mapping ${\rm Prox}_{\lambda q_{\mu,c}}(x)$ decomposes as:
\begin{equation*}
	\left({\rm Prox}_{\lambda q_{\mu,c}}(x)\right)_{I}={\rm Prox}_{\lambda \|\cdot\|_1+\delta_{\mu_I,c}(\cdot)}(x_I),\qquad \left({\rm Prox}_{\lambda q_{\mu,c}}(x)\right)_{I^{\complement}}={\rm Prox}_{\lambda \|\cdot\|_1}(x_{I^{\complement}}),
\end{equation*}
where $I^{\complement}$ is the complement of $I$ in $[n]$, and $\delta_{\mu_I,c}(\cdot)$ is the indicator function of the affine set $C_{\mu_I,c}= \{z\in \mathbb{R}^{|I|}\mid \mu_I^\top z=c\}$. This shows that the coordinates corresponding to indices with $\mu_i = 0$ are unaffected by the affine constraint and can be treated separately using the standard soft-thresholding operator.

We evaluate the proximal operator \eqref{eq: def_prox_qx_C} via its optimality condition, whereby the optimization problem in $\mathbb{R}^n$ is reduced to a one-dimensional root-finding task through the introduction of a scalar dual variable. This approach was previously considered in \cite{pourkamali2020efficient} for the special case $\mu = e$ and $c = 1$. For completeness, and to prepare for our analysis of the B-subdifferential of ${\rm Prox}_{\lambda q_{\mu,c}}(\cdot)$, we extend the result to arbitrary $\mu \in \mathbb{R}^n$ and $ c \in \mathbb{R}$. The extension is conceptually straightforward but serves as a useful basis for the subsequent analysis.

We begin by presenting a characterization of the optimality condition for evaluating the proximal operator in \eqref{eq: def_prox_qx_C}.
\begin{proposition}
	A necessary and sufficient optimality condition for \eqref{eq: def_prox_qx_C} is the existence of a dual multiplier $w\in \mathbb{R}$ such that
	\begin{equation}\label{eq:prox_qx_C_opt}
		f(x, w) := \mu^\top {\rm Prox}_{\lambda \|\cdot\|_1}(x-w \mu) =c. 	
	\end{equation}
	Once such a scalar $w$ is identified, then
	\begin{equation}
		{\rm Prox}_{\lambda q_{\mu,c}}(x) = {\rm Prox}_{\lambda \|\cdot\|_1}(x-w \mu). \label{eq: prox_qx_C}
	\end{equation}
\end{proposition}
\begin{proof}{Proof}
	According to \cite[Corollary 28.3.1]{rockafellar1997convex}, a point $z\in \mathbb{R}^n$ is the minimizer to the optimization problem in \eqref{eq: def_prox_qx_C} if and only if there exists a scalar $w\in \mathbb{R}$ such that the following Karush-Kuhn-Tucker conditions hold:
	\begin{equation}\label{eq: kkt}
		\left\{ \begin{aligned}
			&0\in z-x +\lambda \partial \|z\|_1  + w \mu, \\
			&\mu^\top z = c.
		\end{aligned}
		\right.
	\end{equation}
	Note that the first condition in \eqref{eq: kkt} is equivalent to $
	z = {\rm Prox}_{\lambda \|\cdot\|_1}(x-w \mu)$.
	By plugging it into the affine constraint $\mu^\top z =c$, we have the equality \eqref{eq:prox_qx_C_opt}, and the remaining conclusion follows. 
\end{proof}

The following proposition analyzes the existence of a dual multiplier satisfying the condition \eqref{eq:prox_qx_C_opt}, and also shows that  one such multiplier and the proximal mapping ${\rm Prox}_{\lambda q_{\mu,c}}(x)$ can be computed in  $\mathcal{O}(n \log n)$ operations.

\begin{proposition} \label{prop:wx-exist}
	For any $x \in \mathbb{R}^n$, there must exist some $w\in \mathbb{R}$ such that \eqref{eq:prox_qx_C_opt} holds. Such a scalar $w$ and the proximal mapping
	${\rm Prox}_{\lambda q_{\mu,c}}(x)$
	can be computed in $\mathcal{O}(n \log n)$ arithmetic operations.
\end{proposition}
\begin{proof}{Proof}
	Define the function $s:\mathbb{R}^2\rightarrow \mathbb{R}$ as: 
	\begin{equation*}
		s(t,r) := {\rm sign}(t-r)\left(|t-r|-\lambda\right)_+,\quad t, r\in \mathbb{R}.
	\end{equation*}
	For fixed $t$, the mapping $r\mapsto s(t,r)$ is non-increasing. Consequently, for any $x\in \mathbb{R}^n$, 
	\begin{equation}
		w \mapsto f(x,w)=\sum_{i=1}^{n}\mu_i {\rm Prox}_{\lambda |\cdot|} (x_i-w\mu_i)=\sum_{i=1}^{n}\mu_i s(x_i,w\mu_i) \label{eq: f_in _s}
	\end{equation}
	is a continuous, piecewise affine, {non-increasing function, since $\mu_i\neq 0$ for all $i\in [n]$}. Moreover, as $w \rightarrow -\infty$ (respectively, $\infty$), $f(x,w)\rightarrow \infty$ (respectively, $-\infty$). By the intermediate value theorem, there must exist some $w$ such that $f(x,w) = c$. 
	
	Moreover, given $x\in \mathbb{R}^n$, solving for $w$ such that $f(x,w) = c$ reduces to finding the root of $f(x,\cdot)$. This function changes its slope at $2n$ break-points $\left\{ \frac{x_i \pm \lambda}{\mu_i}, i\in[n] \right\}$, which partition the domain into linear regions. Within each region, the function is affine, so the root can be easily determined once the correct region is identified.
	
	To do this efficiently, we first sort the $2n$ break-points, which takes $\mathcal{O}(n \log n)$ operations. Then, as $f(x,\cdot)$ is monotone, we can apply a bisection search over these regions to locate the interval containing the root. Each bisection step requires $\mathcal{O}(n)$ time, and the total number of steps is $\mathcal{O}(\log n)$. Once the correct region is found, we choose any point $\bar{w}$ in this region to compute 
	\begin{equation*}
		\bar{z} = {\rm Prox}_{\lambda \|\cdot\|_1}(x - \bar{w} \mu) = (s(x_1,\bar{w} \mu_1),\cdots,s(x_n,\bar{w} \mu_n))^\top.
	\end{equation*}	
	Then we have that the support of $\bar{z}$, denoted by $S =\{i \in [n]\mid \bar{z}_i \neq 0\}$, coincides with that of ${\rm Prox}_{\lambda q_{\mu,c}}(x)$. If $S = \emptyset$, then the proximal mapping ${\rm Prox}_{\lambda q_{\mu,c}}(x) = 0$, and any value of $w$ within the identified region is a valid dual multiplier. Otherwise, we have
	$${\rm sign}(\bar z_i) = s(x_i, w^*\mu_i) \quad \mbox{ and } \quad 
	f(x,w^*) = \sum_{i\in S} \mu_i (x_i - w^*\mu_i - {\rm sign}(\bar{z}_i)\lambda).   $$
	This, together with the constraint $f(x,w^*) = c$, gives
	$$
	w^*  =  \dfrac{c - \sum_{i \in S}\mu_i(x_i - {\rm sign}(\bar{z}_i) \lambda)}{\sum_{i\in S} \mu_i^2},
	$$
	and ${\rm Prox}_{\lambda q_{\mu,c}}(x) = {\rm Prox}_{\lambda \|\cdot\|_1}(x - w^* \mu)$. The overall complexity is thus $\mathcal{O}(n \log n)$. 
\end{proof}

For clarity and completeness, we summarize the above procedure to compute the proximal mapping ${\rm Prox}_{\lambda q_{\mu,c}}(\cdot)$ in Algorithm \ref{alg: prox}.

\begin{algorithm}
	\caption{Computation of the proximal mapping ${\rm Prox}_{\lambda q_{\mu,c}}(\cdot)$} 
	\label{alg: prox}
	\vskip6pt
	\begin{algorithmic}
		\STATE \textbf{Input:} $x \in \mathbb{R}^n$, $\lambda>0$, $c\in \mathbb{R}$, and $\mu\in \mathbb{R}^n$ with $\mu_i\neq 0$ for all $i\in [n]$.
		\STATE Let $y$ be the sorted list (in ascending order) of the $2n$ breakpoints $\left\{\frac{x_i \pm \lambda_i}{\mu_i} \right\}_{i=1}^n$
		\STATE Append $y_0 = -\infty$ and $y_{2n+1} = +\infty$, and initialize $i_{\min} \gets 0$, $i_{\max} \gets 2n + 1$
		\WHILE{$i_{\max} - i_{\min} > 1$}
		\STATE $j \gets \lfloor (i_{\min} + i_{\max}) / 2 \rfloor$
		\STATE \textbf{if} $f(x, y_j) > c$ \textbf{then} $i_{\min} \gets j$ \textbf{else} $i_{\max} \gets j$
		\ENDWHILE
		\STATE Compute $\bar{z} = {\rm Prox}_{\lambda \|\cdot\|_1}(x - (y_{i_{\min}} + y_{i_{\max}})\mu/2)$ and let $S =\{i \in [n]\mid \bar{z}_i \neq 0\}$.
		\IF{$S = \emptyset$}
		\STATE\textbf{Output:} ${\rm Prox}_{\lambda q_{\mu,c}}(x) = 0$
		\ELSE		
		\STATE$w^*  =  \dfrac{c - \sum_{i \in S}\mu_i(x_i - {\rm sign}(\bar{z}_i) \lambda)}{\sum_{i\in S} \mu_i^2}$
		\STATE\textbf{Output:} ${\rm Prox}_{\lambda q_{\mu,c}}(x) = {\rm Prox}_{\lambda \|\cdot\|_1}(x - w^* \mu)$
		\ENDIF
	\end{algorithmic}
\end{algorithm}

\section{Characterization of the B-subdifferential of ${\rm Prox}_{\lambda q_{\mu,c}}(\cdot)$}
\label{sec: Bdiff}

In this section, we study the B-subdifferential of ${\rm Prox}_{\lambda q_{\mu,c}}(\cdot)$. As we will see in the subsequent analysis, the differentiability properties of ${\rm Prox}_{\lambda q_{\mu,c}}(\cdot)$ depend crucially on whether $c$ is zero. We begin with the case $c \neq 0$, where the dual multiplier associated with the affine constraint is uniquely defined, and we can characterize both its Lipschitz continuity and the resulting B-subdifferential of ${\rm Prox}_{\lambda q_{\mu,c}}(\cdot)$. These results are presented in Sections \ref{sec: multiplier} to \ref{sec: bdiff_prox}. The case $c = 0$ is addressed in Section \ref{sec: bdiff_c0}, where the analysis is more delicate due to the potential loss of uniqueness and continuity of the multiplier.

\subsection{Lipschitz continuity of the dual multiplier}
\label{sec: multiplier}
We first assume $c\neq 0$. Under this condition, the constraint $\mu^\top x=c$ uniquely determines the dual multiplier $w$, as characterizated in the following proposition.

\begin{proposition} \label{prop:wx-unique}
	Suppose $c\neq 0$. For any $x \in \mathbb{R}^n$, there exists a unique $w\in \mathbb{R}$, which we denote as $w = w(x) $, such that \eqref{eq:prox_qx_C_opt} holds.
\end{proposition}
\begin{proof}{Proof}
	The existence follows directly from Proposition \ref{prop:wx-exist}. 
	We prove the uniqueness by contradiction. 
	Now suppose $w_1<w_2$, with both satisfying \eqref{eq:prox_qx_C_opt}. Then, we have $f(x,w)=c$ for all $w \in [w_1, w_2]$ due to the monotonicity of $f(x,\cdot)$, which further implies $\sum_{i=1}^n \mu_i [s(x_i, w_1\mu_i) - s(x_i, w_2\mu_i)] = 0$ according to \eqref{eq: f_in _s}. Since $\mu_is(x_i, w_1\mu_i) \geq  \mu_i s(x_i, w_2\mu_i)$ and $\mu_i\neq 0$ for $i\in [n]$, we must have $ s(x_i,w_1\mu_i) =  s(x_i,w_2\mu_i)$ for all $i\in [n]$. The latter can happen only when 
	\begin{equation*}
		w_1\mu_i, w_2\mu_i \subseteq \left[x_i-\lambda, x_i+\lambda\right],\quad \mbox{for } i\in [n]. 
	\end{equation*}
	In particular, this implies $s(x_i,w_1\mu_i)= 0$ for $i\in [n]$, and hence $f(x,w_1)=\sum_{i=1}^n \mu_i s(x_i, w_1\mu_i) = 0$, which contradicts the fact that $f(x, w_1)=c\neq 0$. Therefore, the dual multiplier must be unique. 
\end{proof}

For any $x$ and $w(x)$ satisfying equation $\eqref{eq:prox_qx_C_opt}$, we define the following index sets:
\begin{equation}
	\label{def: alpha_beta_gamma}
	\begin{aligned}
	\alpha_+(x) &= \left\{
	i\in [n]\mid x_i - w(x) \mu_i > \lambda
	\right\}, \quad
	\alpha_-(x) = \left\{
	i\in [n]\mid x_i - w(x) \mu_i <- \lambda
	\right\},\notag
	\\
	\gamma(x) &= \left\{ i\in[n]\mid |x_i - w(x) \mu_i| <\lambda \right\},
	\\
	\beta_+(x) &= \left\{
	i\in [n] \mid x_i - w(x) \mu_i = \lambda 
	\right\}, \quad
	\beta_-(x) = \left\{
	i\in [n] \mid x_i - w(x) \mu_i = -\lambda
	\right\}, \notag
	\end{aligned}
\end{equation}
and 
$\alpha(x)= \alpha_+(x) \cup \alpha_-(x)$, $\beta(x) = \beta_+(x) \cup \beta_-(x)$. We note that $\alpha(x)\cup \beta(x) \cup \gamma(x)$ is a partition of the index set $[n]$.  Clearly, since $f(x, w(x)) = c\neq 0$, there is at least one $i\in [n]$ such that $|x_i - w(x)\mu_i|> \lambda$, that is, $
\alpha(x)\neq \emptyset $ 
for any $x\in \mathbb{R}^n$. 

We show in the following proposition that the dual multiplier map $w(\cdot)$ is convex, Lipschitz continuous, and piecewise affine, which will be used when characterizing sensitivity and stability of the dual multiplier map as well as the proximal mapping.
\begin{proposition}\label{prop:wx_continuous}
	Suppose $c\neq 0$. The mapping $w(\cdot)$ defined in Proposition \ref{prop:wx-unique} is convex, Lipschitz continuous, and piecewise affine. 
\end{proposition}
\begin{proof}{Proof}
	For any $x\in \mathbb{R}^n$ and any $w \in \mathbb{R}$, by slightly abusing the notation, we define the following index sets
	\begin{equation*}
		\alpha_+(x, w) = \{ i\in [n] \mid x_i - w\mu_i  > \lambda \}, \qquad
		\alpha_-(x, w) = \{ i\in [n] \mid x_i - w\mu_i  < - \lambda \},
	\end{equation*}
	and $\alpha(x, w)=\alpha_+(x, w)\cup\alpha_-(x, w)$.
	
	\textbf{(i) Continuity.} First, we show continuity of $w(\cdot)$ at any given $x\in \mathbb{R}^n$. Denote 
	\begin{equation*}
		\epsilon_0  = \frac{1}{2} \min_{i\in \alpha(x)} \left\{ \left|\frac{x_i}{\mu_i} - w(x)\right| -\frac{\lambda}{|\mu_i|} \right\} > 0.
	\end{equation*}
	We claim that for any $\epsilon\in (0, \epsilon_0)$ and any $x'\in \mathbb{R}^n$ such that $\|x-x'\|_1\leq  \frac{\epsilon \mu_{\min}^2}{2\mu_{\max}}$, it holds that $|w(x)-w(x')| < \epsilon$. Here $\mu_{\max}:=\max_{i\in [n]}|\mu_i|$ and $\mu_{\min}:=\min_{i\in [n]}|\mu_i|$. We prove this by contradiction. Suppose instead that $|w(x)-w(x')| \geq \epsilon$.
	
	Recall that, $f(x, \cdot)$ is a continuous piecewise linear function, whose slope at $w$ is given by $-\sum_{i\in\alpha(x,w)} \mu_i^2$. Moreover, for any $w' \in [w(x) - \epsilon_0, w(x) + \epsilon_0]$, we have 
	\begin{equation*}
		\begin{aligned}
		|x_i-w'\mu_i|-\lambda&=|\mu_i|\left(\left|\frac{x_i}{\mu_i} - w'\right| -\frac{\lambda}{|\mu_i|}
		\right)\\
		&\geq |\mu_i|\left(\left|\frac{x_i}{\mu_i} - w(x)\right|-|w(x)-w'| -\frac{\lambda}{|\mu_i|}
		\right)\geq |\mu_i|\epsilon_0>0,
		\end{aligned}
	\end{equation*}
	for each $i\in \alpha(x,w)$. That is to say, $\alpha(x) = \alpha(x, w(x))\subseteq \alpha(x, w')$, which further implies that the slope of $f(x, \cdot)$ has magnitude at least $\sum_{i\in\alpha(x)} \mu_i^2$ at any $w'\in [w(x) - \epsilon_0, w (x)+ \epsilon_0]$. Consequently, by the mean value theorem, we have
	\begin{equation}
		|f(x, w(x)) - f(x, w')| \geq \sum_{i\in\alpha(x)} \mu_i^2 \cdot |w(x) - w'| \geq \mu_{\min}^2|w(x) - w'|,\label{eq: control_fix_x}
	\end{equation}
	for any $w' \in [w(x) - \epsilon_0, w(x) + \epsilon_0]$, where the last inequality holds as $\alpha(x)\neq \emptyset$. Then we can see that $|w(x)-w(x')| \geq \epsilon$ indicates 
	\begin{equation}
		|f(x,w(x))-f(x,w(x'))|\geq \mu_{\min}^2\epsilon,\label{eq: muep}
	\end{equation}
	since if $\epsilon\leq |w(x)-w(x')|\leq \epsilon_0$, according to \eqref{eq: control_fix_x}, we have
	\begin{equation*}
		|f(x,w(x))-f(x,w(x'))|\geq \mu_{\min}^2|w(x) - w(x')|\geq \mu_{\min}^2\epsilon;
	\end{equation*}
	and if $|w(x)-w(x')|> \epsilon_0$, by the monotonicity of $ f(x,\cdot)$ and \eqref{eq: control_fix_x}, we have
	\begin{equation*}
		\begin{aligned}
		|f(x,w(x))-f(x,w(x'))|&\geq |f(x,w(x))-f(x,w(x)+{\rm sign}(w(x')-w(x))\epsilon_0)|\\
		&\geq \mu_{\min}^2\epsilon_0\geq \mu_{\min}^2\epsilon.
		\end{aligned}
	\end{equation*}
	Therefore, we can see that
	\begin{equation}
		\begin{aligned}
		|f(x,w(x))-f(x',w(x'))|&\geq |f(x,w(x))-f(x,w(x'))|-|f(x,w(x'))-f(x',w(x'))|\notag \\
		&\geq \mu_{\min}^2\epsilon- \mu_{\max}\|x-x'\|_1\geq\mu_{\min}^2\epsilon/2>0,
		\end{aligned}
		\label{eq:diff_fx_fx'}
	\end{equation}
	where the second inequality follows from \eqref{eq: muep} and
	\begin{equation*}
		\begin{aligned}
		|f(x, w)-f(x', w)| &=\left|\sum_{i\in [n]}\mu_i  {\rm Prox}_{\lambda|\cdot|}(x_i-w\mu_i)-\sum_{i\in [n]}\mu_i{\rm Prox}_{\lambda|\cdot|}(x'_i-w\mu_i)\right|\\
		&\leq \sum_{i\in [n]}|\mu_i||x_i-x'_i|\leq \mu_{\max}\|x-x'\|_1,\qquad w\in \mathbb{R},\quad x, x'\in \mathbb{R}^n.
		\end{aligned}
	\end{equation*}
	The inequality \eqref{eq:diff_fx_fx'} contradicts the fact that $f(x,w(x))=f(x',w(x'))=c$. It follows that $w(\cdot)$ is continuous w.r.t. $\|\cdot\|_1$, and hence w.r.t. $\|\cdot \|_2$. 
	
	\textbf{(ii) Piecewise affine.} Second, we show that $w(\cdot)$ is piecewise affine. For any $x\in \mathbb{R}^n$, since $\alpha(x)\neq \emptyset$, we have
	\begin{equation*}
		\begin{aligned}
		f(x, w(x)) &= \sum_{i\in \alpha_-(x)} \mu_i(x_i-w(x)\mu_i+\lambda) + \sum_{i\in \alpha_+(x)} \mu_i(x_i -w(x)\mu_i-\lambda) \\
		&= \sum_{i\in \alpha(x)} \mu_i x_i -w(x)\sum_{i\in \alpha(x)}\mu_i^2 - \lambda \Bigg(\sum_{i\in \alpha_{+}(x)}\mu_i-\sum_{i\in \alpha_{-}(x)}\mu_i\Bigg) = c, 
		\end{aligned}
	\end{equation*}
	which implies 
	\begin{equation}
		w(x) = \frac{1}{\sum_{i\in \alpha(x)}\mu_i^2} \left(\sum_{i\in \alpha(x)} \mu_i x_i  - \lambda \Bigg(\sum_{i\in \alpha_{+}(x)}\mu_i-\sum_{i\in \alpha_{-}(x)}\mu_i\Bigg) - c \right). \label{eq: wx-expression}
	\end{equation}
	Since there are only finitely many distinct index sets for $\alpha_+(x)$, and $\alpha_-(x)$, it must be the case that $w(x)$ is piecewise affine. 
	
	\textbf{(iii) Lipschitz continuity.} Note that Lipschitz continuity follows since it is continuous and piecewise affine with bounded coefficients, as \eqref{eq: wx-expression} shows. 
	
	\textbf{(iv) Convexity.} Finally, the convexity of $w(\cdot)$ is established as follows. Let $x, x' \in \mathbb{R}^n$ and $t \in [0,1]$. Denote $w = w(x)$, $w' = w(x')$, $\bar{x} = t x + (1-t)x'$ and $\bar w = t w + (1-t) w'$. Since $\bar{x}_i- \bar{w} \mu_i= t(x_i- w\mu_i ) + (1-t)(x_i' - w'\mu_i )$ and ${\rm Prox}_{\lambda|\cdot|}(\cdot)$ is convex, we have
	\begin{equation*}
		\begin{aligned}
		f(\bar{x}, \bar{w}) &= \sum_{i=1}^n \mu_i{\rm Prox}_{\lambda|\cdot|}(\bar{x}_i - \bar{w}\mu_i)\\
		&\leq t \sum_{i=1}^n \mu_i {\rm Prox}_{\lambda|\cdot|} (x_i - w\mu_i) + (1-t) \sum_{i=1}^n\mu_i {\rm Prox}_{\lambda|\cdot|}(x'_i-w'\mu_i) \\
		&= t f(x, w) + (1-t)f(x', w') = c=f(\bar{x},w(\bar{x})).
		\end{aligned}
	\end{equation*}
	Since $f(x,\cdot)$ is non-increasing, we have $w(\bar{x}) \leq \bar{w}$ and convexity of $w(\cdot)$ follows. This completes the proof. 
\end{proof}

\subsection{B-subdifferential of the dual multiplier mapping}
\label{sec: bdiff_wx}
Continuing the assumption $c\neq 0$, we examine the differentiability properties of the dual multiplier mapping $w(\cdot)$. To facilitate subsequent analysis, we give the following lemma on the directional derivative of ${\rm Prox}_{\lambda \|\cdot\|_1}$, obtained by straightforward calculation.
\begin{lemma}\label{lemma: proxl1}
	For any $z, h \in \mathbb{R}^n$, let ${\rm Prox}'_{\lambda \|\cdot\|_1}(z;h)$ be the (one-sided) directional derivative of ${\rm Prox}_{\lambda \|\cdot\|_1}$ at point $z$ along direction $h$. Then, it holds that for all $i=1,\ldots, n$,
	\begin{equation*}
		({\rm Prox}'_{\lambda \|\cdot\|_1}(z;h))_i = 
		\begin{cases}
			\max\{0,h_i\} & {\rm if} \ z_i = \lambda, \\
			h_i & {\rm if}\ z_i >\lambda \mbox{ {\rm or} } z_i <-\lambda,\\ 
			0 & {\rm if}\ z_i \in (-\lambda, \lambda), \\
			\min\{0,h_i\} & {\rm if} \ z_i = -\lambda.
		\end{cases}
	\end{equation*}
\end{lemma}

Since $w(\cdot)$ is Lipschitz continuous piecewise affine according to Proposition \ref{prop:wx_continuous}, it is differentiable almost everywhere by Rademacher’s
theorem \cite[Section 9.J]{rockafellar2009variational}. Denote 
\[
D_w:= \left\{x\in\mathbb{R}^n\mid w(\cdot) \mbox{ is differentiable at } x \right\}.
\]
In the next proposition, we show that $w(\cdot)$ is differentiable at $x$ if and only if $\beta(x) = \emptyset$, where $\beta(x)$ is defined in \eqref{def: alpha_beta_gamma} .

\begin{proposition}\label{prop:dbbeta}
	Suppose $c\neq 0$. For any given $x\in\mathbb{R}^n$, $w(\cdot)$ is differentiable at $x$ if and only if the index set $\beta(x)$ given in \eqref{def: alpha_beta_gamma} is empty. In fact, for any $x\in D_w$, the derivative $w'(x)\in\mathbb{R}^{1\times n}$ takes the form as 
	\begin{equation}\label{eq: w_derivative}
		(w'(x))_i  =  \begin{cases}
			\frac{\mu_i}{\sum_{j\in\alpha(x) }\mu_j^2} 
			&{\rm if}\ i\in \alpha(x) ,
			\\[5pt]
			0 & {\rm otherwise}.
		\end{cases}
	\end{equation}
\end{proposition}
\begin{proof}{Proof}
	We prove the equivalence by showing both directions.
	
	\textbf{($\Leftarrow$) Suppose $\beta(x)=\emptyset$.} In this case, a small perturbation on $x$ will not change the index set $\alpha(\cdot), \beta(\cdot)$ and $\gamma(\cdot)$ in \eqref{def: alpha_beta_gamma}. Together with the expression of $w(\cdot)$ in \eqref{eq: wx-expression}, we can see that $w(\cdot)$ is differentiable at $x$.
	
	\textbf{($\Rightarrow$) Suppose $x\in D_w$.} We prove $\beta(x) = \emptyset$ by contradiction. Suppose instead $|\beta(x)|>0$. By the chain rule of the composition of B-differentiable functions \cite[Proposition 3.1.6]{facchinei2003finite}, the equation \eqref{eq:prox_qx_C_opt} implies that, for any $h\in\mathbb{R}^n$, we have
	\begin{equation}
		\label{eq:dif-w}
		\mu^\top {\rm Prox}'_{\lambda \|\cdot\|_1}(x - w(x)\mu;\, h - (w'(x)h)\mu ) = 0.
	\end{equation}
	Denote $w'(x) =[\eta_1, \eta_2,\cdots,\eta_n]$. Recall that $|\alpha(x)| \geq 1 $ when $c\neq 0$. Pick $i \in \alpha(x)$ and choose $h = e_{i}$, that is, the $i$-th standard basis in $\mathbb{R}^n$, then $h - (w'(x)h) \mu=e_i -\eta_i\mu$. Therefore, according to Lemma \ref{lemma: proxl1} and equation \eqref{eq:dif-w}, we have that
	\begin{equation}
		\begin{aligned}
		\sum_{k\in \alpha(x),k\neq i}\mu_k(-\eta_i\mu_k)&+\mu_i(1-\eta_i\mu_i)+\sum_{k\in \beta_+(x)}\mu_k\max\{0,-\eta_i\mu_k\}\notag\\
		&+\sum_{k\in \beta_-(x)}\mu_k\min\{0,-\eta_i\mu_k\}=0.
		\end{aligned}
		\label{eq: hei}
	\end{equation}
	On the other hand, one can obtain the following equation by choosing $h = -e_{i}$:
	\begin{equation*}
		\sum_{k\in \alpha(x),k\neq i}\!\mu_k\eta_i\mu_k+\mu_i(\eta_i\mu_i\!-\!1)+\! \sum_{k\in \beta_+(x)}\!\! \!\mu_k\max\{0,\eta_i\mu_k\}+\! \sum_{k\in \beta_-(x)}\!\!\!\mu_k\min\{0,\eta_i\mu_k\}=0.
	\end{equation*}
	After summing up the above two equations, we have
	\begin{equation*}
		0 = \sum_{k\in \beta_+(x)}\mu_k |\eta_i\mu_k|-\sum_{k\in \beta_-(x)}\mu_k|\eta_i\mu_k|.
	\end{equation*}
	Equation \eqref{eq: hei} further implies that $\eta_{i}\neq 0$, thus the above equality indicates that
	\begin{equation}
		\sum_{k\in \beta_+(x)}\mu_k |\mu_k| =\sum_{k\in \beta_-(x)}\mu_k |\mu_k|.\label{eq: b+b-}
	\end{equation}
	Since $|\beta(x)|>0$, we have $\beta_{+}(x)\neq \emptyset$ or $\beta_-(x)\neq \emptyset$. Without loss of generality, we assume $\beta_{+}(x)\neq \emptyset$. By taking $j\in \beta_{+}(x)$ and choosing $h=e_j$ or $h=-e_j$, we have:
	\begin{equation}
		\begin{aligned}
		&\sum_{k\in \alpha(x)}\mu_k(-\eta_j\mu_k)+\sum_{k\in \beta_+(x),k\neq j}\mu_k\max\{0,-\eta_j\mu_k\}+\mu_j\max\{0,1-\eta_j\mu_j\}\notag \\
		&\qquad \qquad \qquad +\sum_{k\in \beta_-(x)}\mu_k\min\{0,-\eta_j\mu_k\}=0,\notag\\
		& \sum_{k\in \alpha(x)}\mu_k(\eta_j\mu_k)+\sum_{k\in \beta_+(x),k\neq j}\mu_k\max\{0,\eta_j\mu_k\}+\mu_j\max\{0,-1+\eta_j\mu_j\}\notag \\
		&\qquad \qquad \qquad+\sum_{k\in \beta_-(x)}\mu_k\min\{0,\eta_j\mu_k\}=0.
		\end{aligned} 
		\label{eq:betaj}
	\end{equation}
	Summing the above two equalities, we have
	\begin{equation*}
		\sum_{k\in \beta_+(x),k\neq j}\mu_k |\eta_j \mu_k|+\mu_j |1-\eta_j\mu_j|-\sum_{k\in \beta_-(x)}\mu_k|\eta_j\mu_k|=0.
	\end{equation*}
	Combing with equation \eqref{eq: b+b-}, we have
	\begin{equation*}
		\mu_j |1-\eta_j\mu_j| -\mu_j |\eta_j\mu_j| = 0.
	\end{equation*}
	Since $\mu_j\neq 0$, we have
	\begin{equation*}
		|1-\eta_j\mu_j|-|\eta_j\mu_j|=0,
	\end{equation*}
	which indicates that $\eta_j\mu_j=1/2$. Substituting this back to \eqref{eq:betaj} gives
	\[
	\sum_{k\in \alpha(x)}\mu_k(\frac{1}{2\mu_j}\mu_k)+\sum_{k\in \beta_+(x),k\neq j}\mu_k\max\{0,\frac{1}{2\mu_j}\mu_k\}+\sum_{k\in \beta_-(x)}\mu_k\min\{0,\frac{1}{2\mu_j}\mu_k\}=0,
	\]
	which means that
	\begin{equation*}
		\sum_{k\in \alpha(x)} \mu_k^2+\sum_{k\in \beta_+(x),k\neq j,\mu_k\mu_j > 0} \mu_k^2+\sum_{k\in \beta_-(x),\mu_k\mu_j < 0} \mu_k^2=0.
	\end{equation*}
	Since $|\alpha(x)|>0$, we arrive at a contradiction. Hence, we must have $|\beta(x)| = 0$. 
	
	Note that the desired expression \eqref{eq: w_derivative} follows directly from the formula in \eqref{eq: wx-expression}. We thus complete the proof of this proposition. 
\end{proof}

Based on Proposition \ref{prop:dbbeta}, we can characterize the B-subdifferential of $w(\cdot)$.
\begin{theorem}
	\label{thm:partialBw}
	Suppose $c\neq 0$. For any $x\in \mathbb{R}^n$, we have the following results.
	\begin{enumerate}[label=(\alph*),leftmargin=2em]
		\item We have
		\begin{equation*}
			\partial_B w(x):= \left\{ \lim_{k\rightarrow \infty} w'(x^k)\ \middle\vert \ x^k \rightarrow x, x^k \in D_w
			\right\}\subseteq {\cal M}(x),
		\end{equation*}
		where ${\cal M}(x)$ is a set of linear operators from $\mathbb{R}^n$ to $\mathbb{R}$ defined as
		\begin{equation*}
			\mathcal{M}(x) = \left\{
			h \in \mathbb{R}^{1 \times n}  \middle| 
			h_i =
			\begin{cases}
				\mu_i / s & i \in \alpha(x) \\
				0 & i \in \gamma(x) \\
				0\ {\rm or}\ \mu_i/s & i \in \beta(x)
			\end{cases},
			\left.
			\begin{array}{l}
				\text{with } s = \sum_{j \in \mathcal{S}(x)} \mu_j^2, \\
				\alpha(x) \subseteq \mathcal{S}(x) \subseteq [n] \setminus \gamma(x)
			\end{array}
			\right.
			\right\}
		\end{equation*}
		with $\alpha(\cdot), \beta(\cdot)$ and $\gamma(\cdot)$ being the index sets given in \eqref{def: alpha_beta_gamma}.
		\item For any $\beta_+'(x)\subseteq\beta_+(x)$ and $\beta_-'(x)\subseteq\beta_-(x)$, we can construct $ h^* \in \partial_B w(x)$ as
		\begin{equation*}
			h^*_i = 
			\begin{cases}
				\frac{\mu_i}{\sum_{j\in\alpha(x)\cup \beta_+'(x)\cup \beta_-'(x) }\mu_j^2} 
				&{\rm if}\ i\in \alpha(x) \cup \beta_+'(x)\cup\beta_-'(x),
				\\[5pt]
				0 & {\rm otherwise}.
			\end{cases}
		\end{equation*} 
		\item It holds that $\partial_B w(x) = {\cal M}(x)$.
	\end{enumerate}
\end{theorem}
\begin{proof}{Proof}
	(a) For any $v\in \partial_B w(x)$, we will show $v\in {\cal M}(x)$. By definition of $\partial_B w(x)$, there exists a sequence $\{x^k\} \subseteq D_w$ such that $x^k\to x$ and $w'(x^k) \to v$. Together with the continuity of $w(\cdot)$ proved in Proposition \ref{prop:wx_continuous}, for $k$ sufficiently large, we have 
	\begin{equation*}
		|(x^k)_i - w(x^k)\mu_i| > \lambda \mbox{ for all } i\in \alpha(x), \quad  \mbox{ and } \
		|(x^k)_i - w(x^k)\mu_i| < \lambda \mbox{ for all } i\in \gamma(x). 
	\end{equation*}
	Meanwhile, for $i\in \beta(x)$, $|(x^k)_i - w(x^k)\mu_i| - \lambda$ converges to $|x_i - w(x)\mu_i|-\lambda = 0$ as $ k \rightarrow \infty$.
	Therefore, we have $\alpha(x)\subseteq \alpha(x^k)$ and $\gamma(x) \subseteq \gamma(x^k)$ for sufficiently large $k$. From Proposition \ref{prop:dbbeta}, we know that we have $\beta(x^k) = \emptyset$ for all $k$, and hence
	\[
	\alpha(x) \subseteq \alpha(x^k) = [n] \setminus \gamma(x^k) \subseteq [n] \setminus \gamma(x),\quad \mbox{for large }k.
	\]
	Moreover, we also have $(w'(x^k))_i = \mu_i/\sum_{j\in\alpha(x^k) }\mu_j^2$ if $i\in \alpha(x^k)$, and $0$ otherwise.
	Since $w'(x^k) \to v$, we can define $s = \lim_{k\to \infty} \sum_{j\in \alpha(x^k)}\mu_j^2$.
	Clearly, since $\mu_i\neq 0$ for all $i\in [n]$, there must exist a set ${\cal S}(x)$ such that $\alpha(x) \subseteq {\cal S}(x)\subseteq [n]\backslash \gamma(x)$ and $s=\sum_{j\in {\cal S}(x)}\mu_j^2$. In addition, we can see that $v_i = \mu_i/s$ for all $i\in \alpha(x)$, $v_i \in \{0, \mu_i/s\}$ for all $i\in \beta(x)$, and $v_i = 0$ for all $i\in \gamma(x)$. That is, $v\in {\cal M}(x)$.
	
	(b) We will show that for such $h^*$, there exists a sequence $\{x^k\} \subseteq D_w$ such that
	$x^k \to x$ and $w'(x^k) \to h^*$. Here, we only need to focus on the nontrivial case where $x\not\in D_w$, that is, $\beta(x)\not=\emptyset$. For $k\geq 1$, define a sequence $\{t^k\} \subseteq \mathbb{R}^n$ as follows: for each $i\in[n]$,
	\begin{equation*}
		(t^k)_i = 
		\begin{cases}
			\lambda\frac{\sum\nolimits_{j\in \beta_-'(x)} {\rm sgn} (\mu_j)-\sum\nolimits_{j\in \beta_+'(x)} {\rm sgn} (\mu_j)}{k|\alpha(x)|\mu_i} & \mbox{ if $ i\in \alpha(x) $}, \\[5pt]
			-\frac{\lambda}{k|\mu_i|} & \mbox{ if  $i\in \beta_+(x)\backslash \beta_+'(x)$ or $i\in \beta_-'(x)$}, \\[5pt]
			\frac{\lambda}{k|\mu_i|} & \mbox{ if  $i\in \beta_+'(x)$ or $i\in \beta_-(x)\backslash\beta_-'(x)$}, \\[5pt]
			0  & \mbox{ if $ i\in \gamma(x) $}.
		\end{cases}
	\end{equation*}
	By choosing $x^k = t^k+x$, there must exist an integer $k_0$, such that for all $k\ge k_0$, 
	\begin{equation}
		|(x^k)_i - w(x)\mu_i|   
		\begin{cases}
			<\lambda &\mbox{for $i\in \gamma(x)\cup( \beta_+(x) \backslash \beta_+'(x))\cup (\beta_-(x) \backslash \beta_-'(x))$},
			\\[5pt]
			> \lambda &\mbox{for $i\in \alpha(x)\cup \beta_+'(x)\cup \beta_-'(x)$} .
		\end{cases}\label{eq:xk_wx}
	\end{equation}
	Moreover, for all $k\geq k_0$, we have
	\begin{equation*}
		\begin{aligned}
		&\mu^\top{\rm Prox}_{\lambda \|\cdot\|_1}(x^k \!- \! w(x)\mu)
		=\!\!\! \!\sum_{i\in \alpha_+(x)} \!\!\!\mu_i(t^k_i\!+\! x_i \!-\! w(x)\mu_i \!-\!\lambda ) +\! \!\!\!\sum_{i\in \alpha_-(x)} \!\!\!\mu_i(t^k_i\!+\! x_i \!- \!w(x)\mu_i \!+\!\lambda ) 
		\\[5pt]
		& \qquad\qquad  + \sum_{i\in \beta_+'(x)}\mu_i (t^k_i + x_i - w(x)\mu_i -\lambda ) + \sum_{i\in \beta_-'(x)} \mu_i(t^k_i + x_i - w(x)\mu_i +\lambda )
		\\[5pt]
		& =\sum_{i\in \alpha_+(x)} \mu_i(x_i - w(x)\mu_i -\lambda ) + \sum_{i\in \alpha_-(x)} \mu_i(x_i - w(x)\mu_i +\lambda ) 
		+ \!\!\sum_{i\in \alpha(x)\cup \beta_+'(x)\cup \beta_-'(x)} \mu_it^k_i
		\\
		& 
		= \sum_{i\in \alpha_+(x)} \mu_i(x_i - w(x)\mu_i -\lambda ) + \sum_{i\in \alpha_-(x)} \mu_i(x_i - w(x)\mu_i +\lambda ) 
		\\
		&= \mu^\top{\rm Prox}_{\lambda \|\cdot\|_1}(x - w(x)\mu) \;=\; c.
		\end{aligned}
	\end{equation*}
	That is, $(x^k,w(x))$ is a solution to equation \eqref{eq:prox_qx_C_opt}.  
	Hence, by the uniqueness of the dual multiplier shown in Proposition \ref{prop:wx-unique}, it holds that for all $k\ge k_0$, $w(x^k) = w(x)$. This, together with \eqref{eq:xk_wx}, further implies that for $k\geq k_0$, we have
	\begin{equation*}
		\begin{aligned}
		&\alpha(x^k) = \alpha(x) \cup \beta_+'(x) \cup \beta_-'(x), \qquad  \beta(x^k) = \emptyset,\\
		&\gamma(x^k) = \gamma(x) \cup  ( \beta_+(x) \backslash \beta_+'(x)) \cup ( \beta_-(x) \backslash \beta_-'(x)).
		\end{aligned}
	\end{equation*}
	Therefore, from Proposition \ref{prop:dbbeta}, we know that for $k\geq k_0$, $x^k\in D_w$ and $w'(x^k) = h^*$. Combining with the fact that $x^k\rightarrow x$, we have $ h^* \in \partial_B w(x)$.
	
	(c) This conclusion follows directly from a simple observation that each and every element in ${\cal M}(x)$ can be represented by appropriately choosing the index sets $\beta'_+(x)\subseteq \beta_+(x)$ and $\beta'_-(x)\subseteq \beta_-(x)$. This completes the proof of the theorem. 
\end{proof}

\subsection{B-subdifferential of ${\rm Prox}_{\lambda q_{\mu,c}}(\cdot)$}
\label{sec: bdiff_prox}
Still assuming $c\neq 0$, we now study the B-subdifferential of the proximal mapping ${\rm Prox}_{\lambda q_{\mu,c}}(\cdot)$. According to \eqref{eq: prox_qx_C} and Propositions \ref{prop:wx-unique} and \ref{prop:wx_continuous}, we have that, for any $x\in\mathbb{R}^n$,
\begin{equation}
	{\rm Prox}_{\lambda q_{\mu,c}}(x) = {\rm Prox}_{\lambda \|\cdot\|_1}(x - w(x)\mu),\label{eq: prox_lambda_q} 
\end{equation}
and ${\rm Prox}_{\lambda q_{\mu,c}}(\cdot)$ is convex, Lipschitz continuous, and piecewise affine over $\mathbb{R}^n$. Define 
\[
D_{\mu,c}:= \left\{x\in\mathbb{R}^n\mid {\rm Prox}_{\lambda q_{\mu,c}}(\cdot) \mbox{ is differentiable at } x \right\}.
\]
We shall prove in the next proposition that $D_{\mu,c} = D_w$. Then it follows from Proposition \ref{sec: bdiff_prox} that $x\in D_{\mu,c}$, $x\in D_w$ and $\beta(x) = \emptyset$ are all equivalent.

\begin{proposition}
	\label{prop:d_prox}
	Suppose $c\neq 0$. For any $x\in\mathbb{R}^n$, ${\rm Prox}_{\lambda q_{\mu,c}}(\cdot)$ is differentiable at $x$ if and only if the index set $\beta(x)  = \emptyset$. In fact, for any $x\in D_{\mu,c}$, it holds that
	\begin{equation}
		{\rm Prox}'_{\lambda q_{\mu,c}}(x) = {\rm Diag}(u) - \frac{1}{\sum_{j\in \alpha(x)}\mu_j^2} \tilde{\mu} \tilde{\mu}^\top, \label{eq:ypx}
	\end{equation}
	where $u\in \mathbb{R}^n$ is defined as: $u_i = 1$ for $i\in \alpha(x)$ and $ 0$ otherwise, and $\tilde{\mu}={\rm Diag}(u)\mu$.
\end{proposition}

\begin{proof}{Proof}
	\textbf{($\Leftarrow$) Suppose $\beta(x)=\emptyset$.} For $x\in\mathbb{R}^n$ with $\beta(x) = \emptyset$, we know from Proposition \ref{prop:dbbeta} that $w(\cdot)$ is differentiable at $x$. Meanwhile, the definition of $\beta(x)$ in \eqref{def: alpha_beta_gamma} in  further implies that ${\rm Prox}_{\lambda \|\cdot\|_1}(\cdot)$ is differentiable at $x - w(x)\mu$. Thus, as the composition of $w(\cdot)$ and ${\rm Prox}_{\lambda \|\cdot\|_1}(\cdot)$, ${\rm Prox}_{\lambda q_{\mu,c}}(\cdot)$ in \eqref{eq: prox_lambda_q} is differentiable at $x$.
	
	\textbf{($\Rightarrow$) Suppose $\hat{x}\in D_{\mu,c}$.} We prove $\beta(\hat{x})=\emptyset$ by contradiction. Suppose instead $\beta(\hat{x}) \neq \emptyset$. As $c\neq 0$ implies that $\alpha(\hat{x})\neq \emptyset$, we can choose $i_0\in \alpha(\hat{x})$. Without loss of generality, we assume $i_0\in \alpha_+(\hat{x})$. Then there exists a neighborhood $\mathcal{B}$ of $\hat{x}$ such that $i_0 \in \alpha_+(x)$ for all $x\in \mathcal{B}$. Thus, according to \eqref{eq: prox_lambda_q}, for any $x\in \mathcal{B}$, we have 
	$$
	({\rm Prox}_{\lambda q_{\mu,c}}(x))_{i_0} =  x_{i_0} - w(x)\mu_{i_0} - \lambda.
	$$
	Since $\hat{x}\in D_{\mu,c}$, we have that $({\rm Prox}_{\lambda q_{\mu,c}}(x))_{i_0}$ differentiable at $\hat x$, which implies the differentiability of $w(\cdot)$ at $\hat x$. This contradicts   Proposition \ref{prop:dbbeta}.
	
	For any $x\in D_{\mu,c}=D_w$, by the  chain-rule and equation \eqref{eq: prox_lambda_q}, we have 
	\begin{equation*}
		{\rm Prox}'_{\lambda q_{\mu,c}}(x) = {\rm {\rm Diag}}(u) \big(I_n - \mu w'(x)\big).
	\end{equation*}
	According to \eqref{eq: w_derivative} in Proposition \ref{prop:dbbeta}, it holds that
	\begin{equation*}
		{\rm Prox}'_{\lambda q_{\mu,c}}(x) =  {\rm {\rm Diag}}(u) \left(I_n\! -\! \frac{1}{\sum_{j\in \alpha(x)}\mu_j^2} \mu \mu^\top {\rm Diag}(u) \right) = {\rm Diag}(u) \!-\! \frac{1}{\sum_{j\in \alpha(x)}\mu_j^2} \tilde{\mu} \tilde{\mu}^\top,
	\end{equation*}
	which complete the proof. 
\end{proof}

Based on the above established results, we characterize the B-subdifferential of the proximal mapping ${\rm Prox}_{\lambda q_{\mu,c}}(\cdot)$ in the next theorem.
\begin{theorem}\label{thm:partialBprox}
	Suppose $c\neq 0$. For any $x\in \mathbb{R}^n$, we have the following results.
	\begin{enumerate}[label=(\alph*),leftmargin=2em]
		\item It holds that
		\[
		\partial_B {\rm Prox}_{\lambda q_{\mu,c}}(x) := \left\{ \lim_{k\rightarrow \infty} {\rm Prox}'_{\lambda q_{\mu,c}}(x^k)\ \middle\vert \ x^k \rightarrow x, x^k \in D_{\mu,c}
		\right\} \subseteq {\cal N}(x),
		\]
		where 
		\begin{equation*}
			{\cal N}(x) =
			\left\{
			{\rm Diag}(u) - \frac{1}{s} \tilde{\mu} \tilde{\mu}^\top \middle\vert\
			\begin{aligned}
				&  u_i = 1 \text{ if } i\in \mathcal{S}(x), \text{and } 0 \text{ otherwise}, \ i\in [n]\\
				&\tilde{\mu}={\rm Diag}(u)\mu, \quad s=  \sum\nolimits_{j\in {\cal S}(x)}\mu_j^2,\\
				&\alpha(x) \subseteq {\cal S}(x)\subseteq [n]\backslash \gamma(x) \\
			\end{aligned}
			\right\}.
		\end{equation*}
		\item For any subsets $\beta_+'(x)\subseteq \beta_+(x)$ and $\beta_-'(x)\subseteq \beta_-(x)$,
		define $u^*\in \mathbb{R}^n$ as $(u^*)_i = 1$ if $i\in  \alpha(x) \cup \beta_+'(x)\cup\beta_-'(x)$, and $0$ otherwise, and let $s^* = \sum_{j\in\alpha(x)\cup \beta_+'(x)\cup \beta_-'(x) }\mu_j^2$, $\mu^* = {\rm Diag}(u^*) \mu$. Then, we have
		\[ 
		{\rm Diag}(u^*) - \frac{1}{s^*} \mu^* (\mu^*)^\top  \in \partial_B {\rm Prox}_{\lambda q_{\mu,c}}(x).
		\]
		\item We have that $\partial_B {\rm Prox}_{\lambda q_{\mu,c}}(x) = {\cal N}(x)$.
	\end{enumerate}
\end{theorem}
\begin{proof}{Proof}
	(a) For any $Q\in \partial_B {\rm Prox}_{\lambda q_{\mu,c}}(x)$, we will show that $Q\in {\cal N}(x)$. According to Proposition \ref{prop:d_prox}, we know that $D_{\mu,c} = D_w$. From the definition of $\partial_B {\rm Prox}_{\lambda q_{\mu,c}}(x)$, there exists a sequence $\{x^k\} \subseteq D_{\mu,c} = D_w$ such that $x^k\to x$ and ${\rm Prox}'_{\lambda q_{\mu,c}}(x^k) \to Q$. From Proposition \ref{prop:d_prox}, we know that
	\begin{equation}
		{\rm Prox}'_{\lambda q_{\mu,c}}(x^k) ={\rm Diag}(u^k) - \frac{1}{\sum_{j\in \alpha(x^k)}\mu_j^2} \tilde{\mu}^k (\tilde{\mu}^k)^\top , \label{eq: prox_diff}
	\end{equation}
	where $(u^k)_i = 1$ for $i\in \alpha(x^k)$, $0$ for $i\in \gamma(x^k) = [n]\backslash \alpha(x^k)$, and $\tilde{\mu}^k={\rm Diag}(u^k)\mu$. Similarly as in Theorem \ref{thm:partialBw}, by defining $s = \lim_{k\to \infty} \sum_{j\in \alpha(x^k)}\mu_j^2$, we can see that there must exist a set ${\cal S}(x)$ such that $\alpha(x) \subseteq {\cal S}(x)\subseteq [n]\backslash \gamma(x)$ and $s=\sum_{j\in {\cal S}(x)}\mu_j^2$. Define $u\in \mathbb{R}^n$ as $u_i = 1$ if $i\in \mathcal{S}(x)$, and $0$ otherwise. Then we further have $Q = {\rm Diag}(u) - \frac{1}{s}\tilde{\mu} \tilde{\mu}^\top$ with $\tilde{\mu}={\rm Diag}(u)\mu$. That is, $Q\in {\cal N}(x)$.
	
	Part (b) can be obtained via the same construction as in part (b) of Theorem \ref{thm:partialBw}, and Part (c) follows directly by combining (a) and (b).  
\end{proof}

\subsection{Discussion of the case $c = 0$}
\label{sec: bdiff_c0}
In this subsection, we focus on the case $c =0$. Unlike the case $c\neq 0$, the dual multiplier $w$ may not be unique for a given $x\in \mathbb{R}^n$, as discussed in the following proposition.

\begin{proposition}
	\label{prop: c0_prox}
	Suppose $c=0$. For any $x\in \mathbb{R}^n$, define 
	\begin{equation}
		E_{L}(x) = \max_{i\in [n]}\ \left(\frac{x_i}{\mu_i} - \frac{\lambda}{|\mu_i|}\right),\quad 
		E_{R}(x) = \min_{i\in [n]}\ \left(\frac{x_i}{\mu_i} + \frac{\lambda}{|\mu_i|}\right). 
		\label{eq: def_endpoints}
	\end{equation}
	We have the following conclusions.
	\begin{itemize}[leftmargin=2em]
		\item[(i)] If $E_{L}(x) > E_{R}(x)$, then there exists a unique dual multiplier $w$ which satisfies $f(x,w)=0$, as defined in \eqref{eq:prox_qx_C_opt}.
		\item[(ii)] If $E_{L}(x) \leq E_{R}(x)$, then ${\rm Prox}_{\lambda q_{\mu,c}}(x) =0$.
	\end{itemize}
\end{proposition}
\begin{proof}{Proof}
	(i) The existence is guaranteed by Proposition \ref{prop:wx-exist}, it remains to prove the uniqueness, which we establish via contradiction. Suppose there exist $w_1<w_2$ both satisfying \eqref{eq:prox_qx_C_opt}. Then, by the argument in the proof of Proposition \ref{prop:wx-unique}, we have
	\begin{equation}
		w_1, w_2 \subseteq J_i:=\left[\frac{x_i}{\mu_i}-\frac{\lambda}{|\mu_i|}, \frac{x_i}{\mu_i}+\frac{\lambda}{|\mu_i|}\right],\quad \mbox{for } i\in [n],  \label{eq: def_Ji}
	\end{equation}
	which contradicts $E_{L}(x) > E_{R}(x)$. Hence, the dual multiplier $w$ is unique. 
	
	(ii) If $E_{L}(x) \leq E_{R}(x)$, then $\bigcap_{i=1}^n J_i$ is non-empty, where the set $J_i$ is defined in \eqref{eq: def_Ji}. For any $w$ in this intersection set, we have $ {\rm Prox}_{\lambda |\cdot|}(x_i-w \mu_i) = 0$ for all $i\in[n]$. This, together with \eqref{eq: prox_qx_C}, implies ${\rm Prox}_{\lambda q_{\mu,c}}(x) = {\rm Prox}_{\lambda \|\cdot\|_1}(x-w \mu)=0$. 
\end{proof}

Based on the above results, we state the following theorem on the B-subdifferential of ${\rm Prox}_{\lambda q_{\mu,c}}(\cdot)$.
\begin{theorem}\label{thm: c0Bdiff_prox}
	Suppose $c=0$. For any $x\in \mathbb{R}^n$, we have
	\begin{equation*}
		\partial_B \mathrm{Prox}_{\lambda q_{\mu,c}}(x) 
		\begin{cases}
			=\mathcal{N}(x), & \text{if } E_L(x) > E_R(x), \\
			=\{0_{n\times n}\} , & \text{if } E_L(x) < E_R(x), \\
			\ni \{0_{n\times n}\} , & \text{otherwise}.
		\end{cases}
	\end{equation*}
	where ${\cal N}(\cdot)$ is defined in Theorem \ref{thm:partialBprox}.
\end{theorem}
\begin{proof}{Proof}
	When $E_{L}(x) > E_{R}(x)$, from Proposition \ref{prop: c0_prox}, we know that there exists a unique multiplier $w$ such that \eqref{eq:prox_qx_C_opt} holds. Since the set $\{x\in \mathbb{R}^n \mid E_{L}(x) > E_{R}(x)\}$ is open, we can apply the same reasoning as in Sections \ref{sec: bdiff_wx} and \ref{sec: bdiff_prox} to conclude that $\partial_B {\rm Prox}_{\lambda q_{\mu,c}}(x) = {\cal N}(x)$. The details are analogous and omitted for brevity.
	
	If $E_{L}(x) < E_{R}(x)$, we know from Proposition \ref{prop: c0_prox} that ${\rm Prox}_{\lambda q_{\mu,c}}(x) =0$. This means that ${\rm Prox}_{\lambda q_{\mu,c}}(\cdot)$ is locally constant in the open set $\{x\in \mathbb{R}^n \mid E_{L}(x) < E_{R}(x)\}$, thus it is differentiable with  ${\rm Prox}'_{\lambda q_{\mu,c}}(x) = 0_{n\times n}$.
	
	Lastly, we consider the case when $E_{L}(x) = E_{R}(x)$. Denote the set 
	$$
	\Omega = \underset{i\in [n]}{\arg\max} \left(x_i/\mu_i - \lambda/|\mu_i|\right),
	$$
	and define a sequence $\{t^k\}\subseteq \mathbb{R}^n$ as
	\begin{equation*}
		(t^k)_i = 
		\begin{cases}
			- \lambda{\rm sign}(\mu_i)/k & \mbox{ if $ i\in \Omega $} \\
			0  & \mbox{ if $ i\in  [n]\backslash \Omega $} \\
		\end{cases}.
	\end{equation*}
	Then for $x^k:=x+t^k$, we have $x^k\rightarrow x$. Moreover, for any $k\geq 1$ and $i\in G$,
	\begin{equation*}
		\begin{aligned}
		&\frac{(x^k)_i}{\mu_i} - \frac{\lambda}{|\mu_i|}  = \frac{x_i}{\mu_i} - \frac{\lambda}{|\mu_i|} - \frac{\lambda}{k|\mu_i|} < \frac{x_i}{\mu_i} - \frac{\lambda}{|\mu_i|}  = E_L(x),\\
		&\frac{(x^k)_i}{\mu_i} + \frac{\lambda}{|\mu_i|}  = \frac{x_i}{\mu_i} + \frac{\lambda}{|\mu_i|} - \frac{\lambda}{k|\mu_i|}\geq \frac{x_i}{\mu_i}-\frac{\lambda}{|\mu_i|} = E_L(x)  = E_R(x).
		\end{aligned}
	\end{equation*}
	This means for any $k\geq 1$, we have $
	E_{L}(x^k) <E_L(x) =E_R(x)\leq  E_{R}(x^k)$,
	and thus ${\rm Prox}'_{\lambda q_{\mu,c}}(x^k) = 0_{n\times n}$. Therefore, we can see that
	$ 0_{n\times n} \in \partial_B {\rm Prox}_{\lambda q_{\mu,c}}(x) $. 
\end{proof}

\section{Double-loop algorithm for affine-constrained sparse optimization}
\label{sec: alg}
In this section, we apply the preconditioned proximal point algorithm (PPA) to solve the optimization problem \eqref{eq: l1_linear_constraint_obj}, which combines a general loss function with an affine-constrained $\ell_1$-regularization term. The success of PPA in large-scale nonsmooth optimization depends crucially on the efficient solutions of a sequence of subproblems. Building on the insights into the B-subdifferential of ${\rm Prox}_{\lambda q_{\mu,c}}(\cdot)$ established in the previous section, we develop a semismooth Newton-type method tailored to the subproblems. This yields a double-loop algorithm capable of tackling \eqref{eq: l1_linear_constraint_obj} with high efficiency and accuracy.

\subsection{Outer loop: preconditioned PPA}
For the problem \eqref{eq: l1_linear_constraint_obj}, the preconditioned PPA generates a sequence \( \{x^k\} \) by solving:
\begin{equation}
	x^{k+1} \approx \underset{x \in \mathbb{R}^n}{\arg\min} \left\{F_k(x):= F(x) + \frac{1}{2\sigma_k} \|x - x^k\|^2 +  \frac{\tau}{2\sigma_k} \|Ax - Ax^k\|^2 \right\}, \label{eq: ppa_subproblem}
\end{equation}
where $\tau > 0$ is a preset constant and $\{\sigma_k\}$ is a nondecreasing sequence of positive real numbers. Any such choice guarantees the convergence of iterates \eqref {eq: ppa_subproblem} established in Theorem \ref{thm:convergence_PPA}; in our implementation, we take $\tau = 1/\lambda_{\max}(A A^\top)$ and $\sigma_k = 3^{\lfloor k/2 \rfloor}$.

To design an efficient algorithm to solve the affine-constrained nonsmooth PPA subproblem \eqref{eq: ppa_subproblem}, the following results \cite{moreau1965proximite,rockafellar1976monotone} regarding the proximal mapping and the Moreau envelope will be useful. 
\begin{lemma}\label{lemma:prox-known-facts}
	For a closed proper convex function $h(\cdot)$, its Moreau envelope is 
	$$
	{\rm E}_{\alpha h}(x) = \min_u \left\{ \alpha h(u)+\frac{1}{2}\|u-x\|^2\right\},\quad \mbox{for any }\alpha >0.
	$$
	The strong convexity of the objective ensures that it is well defined and admits a unique minimizer, denoted as ${\rm Prox}_{\alpha h}(x)$. 
	Moreover, ${\rm E}_{\alpha h}$ is smooth with 
	\[
	\nabla {\rm E}_{\alpha h}(x) = x-{\rm Prox}_{\alpha h}(x) = \alpha {\rm Prox}_{h^*/\alpha}(x/\alpha).
	\] 
	and ${\rm Prox}_{\alpha h}(\cdot)$ is Lipschitz with modulus $1$.
\end{lemma}

Following \cite[Example 11.46]{rockafellar2009variational}, the Lagrangian function associated with \eqref{eq: ppa_subproblem} is 
\begin{equation*}
	\begin{aligned}
	& \ell(x;y) \!=\! \inf_{z\in \mathbb{R}^m} \left\{  
	f(Ax\!-\!z) + \lambda q_{\mu,c}(x) + \frac{1}{2\sigma_k} \|x \!- \!x^k\|^2 +  \frac{\tau}{2\sigma_k} \|Ax\!-\!z \!-\! Ax^k\|^2 + \langle y,z\rangle
	\right\}\\
	& = \inf_{\tilde{z}\in \mathbb{R}^m}\ \left\{  
	f(\tilde{z}) + \lambda q_{\mu,c}(x) + \frac{1}{2\sigma_k} \|x - x^k\|^2 +  \frac{\tau}{2\sigma_k} \|\tilde{z} - Ax^k\|^2 + \langle y,Ax\rangle - \langle \tilde{z},y\rangle
	\right\}\\
	& = \frac{\tau}{\sigma_k} {\rm E}_{\sigma_k f/\tau}\left(Ax^k + \frac{\sigma_k}{\tau} y\right) - \frac{\tau}{2\sigma_k} \|Ax^k + \frac{\sigma_k}{\tau} y\|^2 +\frac{\tau}{2\sigma_k} \|Ax^k\|^2  \\
	&\quad \quad + \lambda q_{\mu,c}(x) + \frac{1}{2\sigma_k} \|x - x^k\|^2 + \langle y,Ax\rangle ,
	\end{aligned}
\end{equation*}
for any $(x,y)\in \mathbb{R}^n\times \mathbb{R}^m$. Then, the dual problem of \eqref{eq: ppa_subproblem} takes the form of 
\begin{equation}
	\max_{y\in \mathbb{R}^m} \left\{ G_k(y):=\min_{x\in \mathbb{R}^n}\ \ell(x;y)\right\} \label{eq: ppa_sub_dual} 
\end{equation}
where 
\begin{equation*}
	\begin{aligned}
	G_k(y)&=\frac{\tau}{\sigma_k} {\rm E}_{\sigma_k f/\tau}\left(Ax^k + \frac{\sigma_k}{\tau} y\right) +
	\frac{1}{\sigma_k} {\rm E}_{\sigma_k \lambda q_{\mu,c}} (x^k - \sigma^k A^\top y)\\
	& - \frac{1}{2\sigma_k} \|x^k - \sigma^k A^\top y\|^2 + \frac{1}{2\sigma_k } \|x^k\|^2 - \frac{\tau}{2\sigma_k} \|Ax^k + \frac{\sigma_k}{\tau} y\|^2+\frac{\tau}{2\sigma_k} \|Ax^k\|^2 .
	\end{aligned}
\end{equation*}
And the Karush–Kuhn–Tucker(KKT) conditions associated with \eqref{eq: ppa_subproblem} and \eqref{eq: ppa_sub_dual} are: 
\begin{equation}
	\left\{\begin{aligned}
		&x = {\rm Prox}_{\sigma_k \lambda q_{\mu,c}} (x^k - \sigma_k A^\top y),  \\
		&Ax = {\rm Prox}_{\sigma_k f/\tau} \left(Ax^k + \frac{\sigma_k}{\tau} y\right) . 
	\end{aligned}
	\right. \label{eq: sub_kkt}
\end{equation}
Based on the relationship \eqref{eq: sub_kkt}, in order to solve each PPA subproblem \eqref{eq: ppa_subproblem}, we only need to solve its dual \eqref{eq: ppa_sub_dual}. 

The next theorem shows the convergence result of the preconditioned PPA iterations with dual-based subproblem solutions, following similar augment as  in \cite{lin2024highly}.
\begin{theorem}\label{thm:convergence_PPA}
	Let $\{(x^k,y^k)\}$ be generated by the preconditioned PPA, where at the $k$-th iteration the subproblem is solved via its dual as:
	\begin{equation}
		\left\{
		\begin{aligned}
			y^{k+1}&\approx  \max_{y\in \mathbb{R}^m}\ G_k(y),\\
			x^{k+1} &= {\rm Prox}_{\sigma_k \lambda q_{\mu,c}} (x^k - \sigma_k A^\top y^{k+1}),
		\end{aligned}
		\right.\label{eq: ppa_iter}
	\end{equation}
	subject to the primal–dual gap condition
	\begin{equation}
		F_k(x^{k+1})-G_k(y^{k+1})\leq \frac{\epsilon_k^2}{2\sigma_k}\min
		\left\{
		1, \|x^{k+1}-x^k\|^2+\tau \|Ax^{k+1}-Ax^k\|^2
		\right\},
	\end{equation}
	where $\{\epsilon_k\}$ is a preset summable nonnegative sequence with $\epsilon_k<1$; in our implementation, we take $\epsilon_k = 0.5/1.06^k$. Denote the optimal solution set to \eqref{eq: l1_linear_constraint_obj} as ${\cal X}^*$. Then we have the following conclusions.
	\begin{enumerate}[label=(\alph*),leftmargin=2em]
		\item The sequence $\{x^k\}$ converges to some point in ${\cal X}^*$.
		\item Denote ${\cal M}:=I_n+\tau A^\top A$. Suppose there exists $\kappa>0$, such that for any $x\in \mathbb{R}^n$ with ${\rm dist}(x,{\cal X}^*)\leq \sum_{i=0}^{\infty}\epsilon_k+{\rm dist}_{{\cal M}}(x^0,{\cal X}^*)$, we have
		\begin{equation*}
			{\rm dist}(x,{\cal X}^*)\leq \kappa {\rm dist}(0,\partial F(x)),
		\end{equation*}
		Then there exists a sequence $\{\theta_k\}$ with $0\leq \theta_k < 1$, such that for all sufficiently large $k$, we have
		\begin{equation*}
			{\rm dist}_{{\cal M}}(x^{k+1},{\cal X}^*)\leq \theta_k {\rm dist}_{{\cal M}}(x^{k},{\cal X}^*).
		\end{equation*}
	\end{enumerate}
\end{theorem}

\begin{remark}
	Note that we allow the PPA parameters $\{\sigma_k\}$ to be dynamically adjusted, potentially based on all past iterates $\{(x^\ell, y^\ell)\}_{l=1}^k$, past parameter values $\{\sigma_\ell\}_{l=1}^k$, as well as running statistics like primal/dual infeasibility norms and duality gap. The above convergence guarantee allows sufficient flexibility for various update schemes. In particular, if $\sigma_k\rightarrow \infty$, the sequence $\{x^k\}$ attains superlinear convergence, meaning that $\{\mu_k\}$ in Theorem \ref{thm:convergence_PPA} tends to zero.
\end{remark}

\subsection{Inner loop: semismooth Newton method}
Note that $G_k(\cdot)$ is concave and smooth, with its optimality condition given by  
\begin{equation}
	\nabla G_k(y) = -{\rm Prox}_{\sigma_k f/\tau} \left(Ax^k + \frac{\sigma_k}{\tau} y\right) + A\, {\rm Prox}_{\sigma_k \lambda q_{\mu,c}}(x^k - \sigma_k A^\top y) = 0.
	\label{eq:opt-condition-subprob}
\end{equation}
In practice, for many commonly used loss functions, such as least squares, logistic, or square-root loss, the proximal mapping ${\rm Prox}_{f} \left(\cdot \right)$ and its Clarke generalized Jacobian $\partial {\rm Prox}_{ f}(\cdot)$ are explicitly computable. Moreover, by our established Algorithm \ref{alg: prox}, Theorems \ref{thm:partialBprox} and \ref{thm: c0Bdiff_prox}, the proximal mapping ${\rm Prox}_{q_{\mu,c}} \left(\cdot \right)$ and its B-subdifferential are also available. Consequently, the optimality condition \eqref{eq:opt-condition-subprob} can be efficiently solved using a semismooth Newton method.

Define the following operator from $\mathbb{R}^m$ to $\mathbb{R}^m$: for any $y\in \mathbb{R}^m$,
\begin{equation*}
	\hat{\partial}^2 G_k(y):= - \frac{\sigma_k}{\tau} \partial {\rm Prox}_{\sigma_k f/\tau} \left(Ax^k + \frac{\sigma_k}{\tau} y\right)  - \sigma_k A\, \partial {\rm Prox}_{\sigma_k \lambda q_{\mu,c}}(x^k - \sigma_k A^\top y) \, A^\top. 
\end{equation*}
Based on the characterization of $\partial_B {\rm Prox}_{q_{\mu,c}}(\cdot)$ in Section \ref{sec: Bdiff}, we can readily construct an element $U_k(y) \in \partial_B {\rm Prox}_{\sigma_k q_{\mu,c}}(x^k - \sigma_k A^\top y)$, which lies within the Clarke subdifferential $\partial {\rm Prox}_{\sigma_k q_{\mu,c}}(x^k - \sigma_k A^\top y)= {\rm conv} \partial_B {\rm Prox}_{\sigma_k q_{\mu,c}}(x^k - \sigma_k A^\top y)$, where ${\rm conv}$ denotes the convex hull. Moreover, if one can select some $H_k(y)\in  \partial {\rm Prox}_{\sigma_k f/\tau} \left(Ax^k + \frac{\sigma_k}{\tau} y\right) $, constructing an element of $\hat{\partial}^2 G_k (y)$ is mathematically straightforward as: 
\begin{equation}
	V_k(y):= - \frac{\sigma_k}{\tau}  H_k(y) - \sigma_k A U_k(y) A^\top \in \hat{\partial}^2 G_k(y). \label{eq: one_hessian}
\end{equation}

With the above construction, we can solve \eqref{eq:opt-condition-subprob} using a semismooth Newton method. The following theorem presents a key result that supports the implementation of this method and establishes its convergence properties.

\begin{theorem}\label{thm:ssn}
	Suppose the equation \eqref{eq:opt-condition-subprob} admits a unique solution, denoted as $\bar{y}$. Assume ${\rm Prox}_{f}(\cdot)$ is strongly semismooth with respect to $\partial {\rm Prox}_{ f}(\cdot)$, and each element in $\hat{\partial}^2 G_k(\bar{y})$ is negative definite. Let $\{y^j\}$ be generated by the semismooth Newton method as follows: at iteration $j$, compute
	\begin{equation*}
		y^{j+1} = y^{j} + \alpha_j d^j,
	\end{equation*}
	where
	\begin{itemize}[leftmargin=2em]
		\item $d^j$ approximately solves 
		\[V_k(y^j)[d^j]-\varepsilon_j d^j \approx - \nabla G_k(y^j) \mbox{ with } \varepsilon_j = 0.1 \min (0.1, \|\nabla G_k(y^{j})\|),\] such that the residual satisfies
		$$
		\|V_k(y^j)[d^j]   -\varepsilon_j d^j + \nabla G_k(y^{j})\| \leq \min(0.005, \|\nabla G_k(y^j)\|^{1+\delta}).
		$$
		\item $\alpha_j = 1/2^{m_j}$, with $m_j$ being the smallest nonnegative integer such that 
		$$
		G_k(y^j +  d^j/2^m) \geq G_k(y^j) + (10^{-4}/2^m) \langle \nabla G_k(y^j), d^j \rangle,
		$$
	\end{itemize}
	here $\delta \in (0,1]$ is a predefined parameter (in our implementation, we set $\delta= 0.5$). Then, we have that $\{y^{j}\}$ converges to $\bar{y}$. Meanwhile, for all $j\geq 1$,
	\[ 
	\|y^{j+1} - \bar{y}\| = {\cal O}
	(\|y^{j} - \bar{y}\|^{1+\delta}). 
	\]
\end{theorem}
\begin{proof}{Proof}
	According to Proposition \ref{prop:wx-exist}, for any $\nu>0$, the operator ${\rm Prox}_{\nu q_{\mu,c}}(\cdot)$ is piecewise affine. By \cite{sun2008lowner}, it is strongly semismooth with respect to $\partial {\rm Prox}_{\nu q_{\mu,c}}(\cdot)$. Then it can be seen that $\nabla G_k(\cdot)$ is strongly semismooth with respect to $\hat{\partial}^2 G_k(\cdot)$. The remaining result follows by arguments similar to those in \cite[Theorem 3.5]{zhao2010newton}. 
\end{proof}

As a side note, if all elements of $\hat{\partial}^2 G_k(y)$ are negative definite for every $y \in \mathbb{R}^m$, then $\varepsilon_j$ can be set to zero for all $j$.

\begin{remark}
	A broad class of standard loss functions satisfies the assumptions in Theorem~\ref{thm:ssn}. Two representative examples are as follows. First, if the loss function $f(\cdot)$ is twice continuously differentiable (e.g., the least squares or logistic loss), the proximal mapping ${\rm Prox}_{f}(\cdot)$ is smooth with a positive-definite gradient, as established in \cite[Proposition 4.1]{lin2024highly}. In this case, the assumptions in Theorem~\ref{thm:ssn} are satisfied. Second, for the square-root loss $f(z) = \|z - b\|$, suppose the regularity condition $A\bar{x} - b \neq 0$ holds, where $\bar{x}$ denotes the unique solution to the PPA subproblem \eqref{eq: ppa_subproblem}. From the KKT system \eqref{eq: sub_kkt}, for any optimal solution $\tilde{y}$ to \eqref{eq:opt-condition-subprob}, we have
	\begin{equation*}
		\begin{aligned}
		A\bar{x} - b &= {\rm Prox}_{\frac{\sigma_k}{\tau} \|\cdot - b\|} \left( Ax^k + \frac{\sigma_k}{\tau}\tilde{y} \right) - b \\
		&= \left( Ax^k + \frac{\sigma_k}{\tau} \tilde{y} - b \right) - \Pi_{\{\|\cdot\| \leq \sigma_k/\tau\}}\left( Ax^k + \frac{\sigma_k}{\tau} \tilde{y} - b \right).
		\end{aligned}
	\end{equation*}
	Since $A\bar{x} - b \neq 0$, it follows that $\|Ax^k + \frac{\sigma_k}{\tau} \tilde{y}  - b\| > \frac{\sigma_k}{\tau}$, which implies that the proximal mapping ${\rm Prox}_{\sigma_k f/\tau}(\cdot)$ is differentiable at $Ax^k + \frac{\sigma_k}{\tau}\tilde{y} $. This differentiability holds for all dual optimal solutions, which implies strong concavity of the dual objective at these points and forces the solutions to coincide. Thus, the dual optimal solution is unique, and the remaining assumptions in Theorem~\ref{thm:ssn} are also satisfied.
\end{remark}

\subsection{Implementation details}
In this subsection, we design a fast and memory-efficient implementation for solving the Newton system, the most computationally demanding component of the proposed double-loop algorithm. 

To illustrate the key ideas more clearly, we consider the least squares loss function, i.e., $f(z) = \|z - b\|^2/2$ with given $b \in \mathbb{R}^m$, as a representative example. In this case, by \cite[Proposition 4.1]{lin2024highly}, we have that for any $y\in \mathbb{R}^m$,
\begin{equation*}
	\partial {\rm Prox}_{\sigma_k f/\tau} \left( y\right) = \left\{ \nabla {\rm Prox}_{\sigma_k f/\tau} \left(y\right) 
	\right\} = \frac{1}{1+\sigma_k/\tau} I_m.
\end{equation*}
Substituting into \eqref{eq: one_hessian} and choosing $\bar{U} \in \partial {\rm Prox}_{\sigma_k q_{\mu,c}}(x^k \!-\! \sigma_k A^\top y)$, we have that
$
- \frac{1}{1+\tau/\sigma_k} I_m - \sigma_k A \bar{U} A^\top \in \hat{\partial}^2 G_k (y).$
The Newton system in Theorem \ref{thm:ssn} then becomes
\begin{equation}
	\left[ \left( \frac{1}{1+\tau/\sigma_k} +\varepsilon_j\right)I_m + \sigma_k A \bar{U} A^\top
	\right] d = \nabla G_k(y^j). \label{eq: newton}
\end{equation}

Based on the fact that $\partial_B {\rm Prox}_{\sigma_k q_{\mu,c}}(x^k \!-\! \sigma_k A^\top y)\subseteq \partial {\rm Prox}_{\sigma_k q_{\mu,c}}(x^k \!-\! \sigma_k A^\top y)$, and Theorems \ref{thm:partialBprox} and \ref{thm: c0Bdiff_prox}, it suffices to consider the case where either $c \neq 0$, or $c = 0$ and $E_L(x) > E_R(x)$; otherwise, setting $\bar{U} = 0_{n \times n}$ makes the Newton system \eqref{eq: newton} trivial to solve. When $c \neq 0$, or $c = 0$ with $E_L(x) > E_R(x)$, we can take
\[
\bar{U} = \operatorname{Diag}(\bar{u}) - \frac{1}{\bar{s}} \bar{\mu} \bar{\mu}^\top\in \partial_B {\rm Prox}_{\sigma_k q_{\mu,c}}(x^k - \sigma_k A^\top y^j) ,
\]
where $\bar{s} = \sum_{i \in \alpha(x^k - \sigma_k A^\top y^j)} \mu_i^2$, and 
\begin{equation*}
	\bar{u}_i = \left\{
	\begin{aligned}
		&1 && \mbox{if } i \in \alpha(x^k - \sigma_k A^\top y^j)\\
		&0 && \mbox{otherwise}
	\end{aligned}
	\right.,\ i \in [n],\qquad \bar{\mu} = \operatorname{Diag}(\bar{u}) \mu,
\end{equation*}
where the index set $\alpha(\cdot)$ is defined in \eqref{def: alpha_beta_gamma}.
Let $K = \alpha(x^k - \sigma_k A^\top y^j)$. Then, the matrix product $A \bar{U} A^\top$ can be computed as
\begin{equation}
	A \bar{U} A^\top = A_K A_K^\top - \frac{1}{\bar{s}} A_K \mu \mu^\top A_K^\top = A_K A_K^\top - \frac{1}{\bar{s}} (A_K \mu)(A_K \mu)^\top, \label{eq: compute_AUAT}
\end{equation}
where $A_K$ denotes the submatrix of $A$ formed by the columns indexed by $K$. Note that $|K|$ is typically much smaller than $n$ due to the sparsity-inducing property of the regularizer $q_{\mu,c}(\cdot)$. Hence, equation \eqref{eq: compute_AUAT} then implies that solving the linear system \eqref{eq: newton} requires ${\cal O}(m^2 |K|)$ operations. Moreover, when $|K|<m$, the cost can be further reduced to ${\cal O}(m |K|^2)$ suing the Sherman–Morrison-Woodbury formula.

\section{Numerical experiments}
\label{sec: exp}
In this section, we demonstrate the effectiveness and scalability of our proposed double-loop algorithm, which leverages the characterization on the B-subdifferential of the affine-constrained $\ell_1$ regularizer
in Section~\ref{sec: Bdiff}. We evaluate the algorithm on two representative applications: microbiome compositional data analysis and sparse subspace clustering. Our experiments also include comparisons with state-of-the-art solvers, highlighting the advantages of the proposed approach.

Our algorithm is implemented in {\sc Matlab}. All experiments were conducted on an Apple M3 system running macOS (version 15.3.1) with 24 GB of RAM.

\subsection{Microbiome compositional data analysis}
We apply our double-loop algorithm to identify key bacterial taxa in the human oral microbiome, and benchmark its performance against existing solvers. 

We downloaded the dataset corresponding to Study ID 14375 from the ORIGINS study  (\url{https://qiita.ucsd.edu/study/description/14375}). The dataset contains microbiome profiles represented as Operational Taxonomic Units (OTUs), which we use as features to predict each sample’s BMI. Each OTU  corresponds to a distinct bacterial species, with counts reflecting its observed abundance within a sample, thereby capturing the microbial composition. After excluding samples with missing BMI data, the final dataset comprises 932 samples and 209,356 OTUs. To model the compositional microbiome data using a log-contrast approach, we first replace zero counts with a small pseudo-count of 0.5. Each sample’s OTU counts are then normalized by its total count and log-transformed for analysis.

To demonstrate the flexibility and effectiveness of our algorithm, we consider two tasks: (1) predicting continuous BMI values via model \eqref{eq: micro_ls}, and (2) classifying samples as above or below the mean BMI using model \eqref{eq: micro_logistic}.

\subsubsection{Regression analysis}
To test the performance of our proposed algorithm for solving \eqref{eq: micro_ls}, we benchmark it against \texttt{SparseReg} (\url{https://github.com/Hua-Zhou/SparseReg}), a state-of-the-art {\sc Matlab} solver for $\ell_1$-regularized least squares problems with linear constraints \cite{gaines2018algorithms}. \texttt{SparseReg} offers three algorithmic options: a quadratic programming approach, an ADMM-based solver, and a path-following algorithm. According to \cite{gaines2018algorithms}, the quadratic programming method yields the poorest performance and is therefore excluded from our comparison. Instead, we compare our algorithm with both ADMM and path-following algorithms, under experimental settings tailored to each method.

\begin{figure}[H]
	{\begin{subfigure}[b]{0.5\textwidth}
			\centering
			\includegraphics[width=\textwidth]{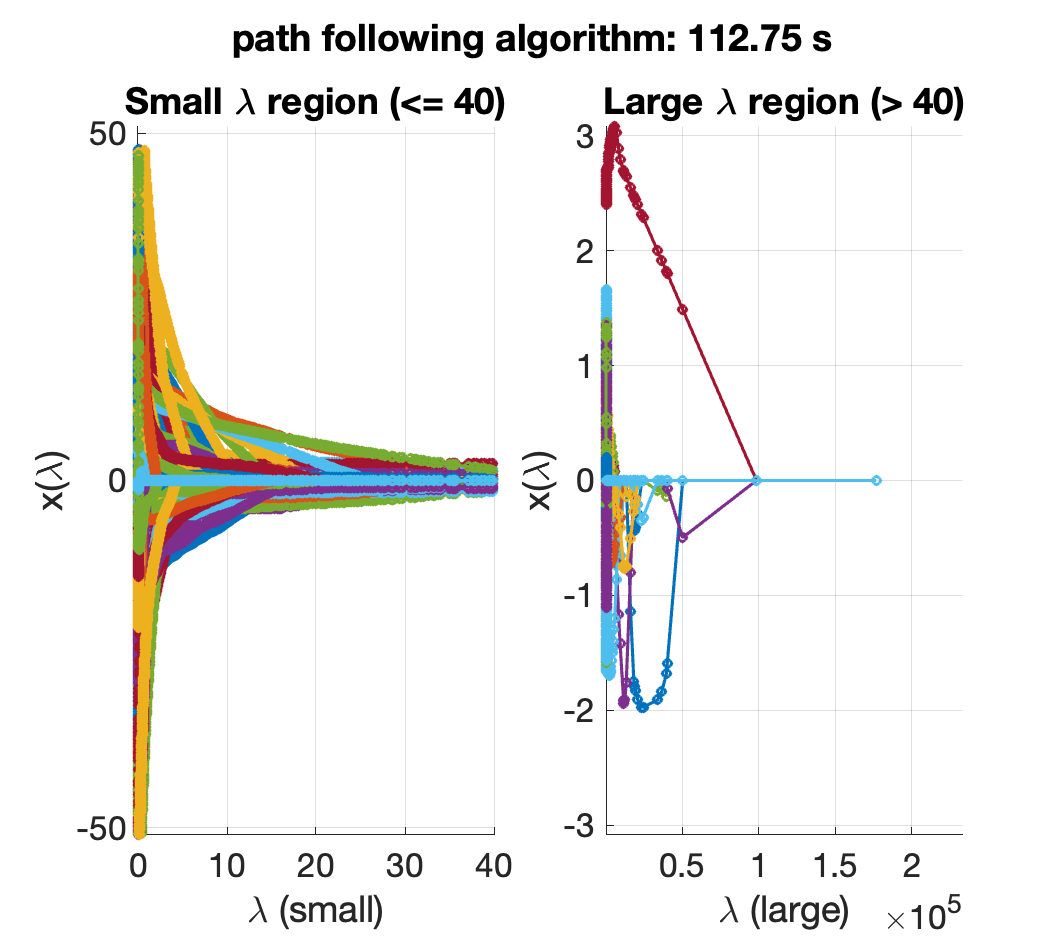}
		\end{subfigure}
		\begin{subfigure}[b]{0.5\textwidth}
			\centering
			\includegraphics[width=\textwidth]{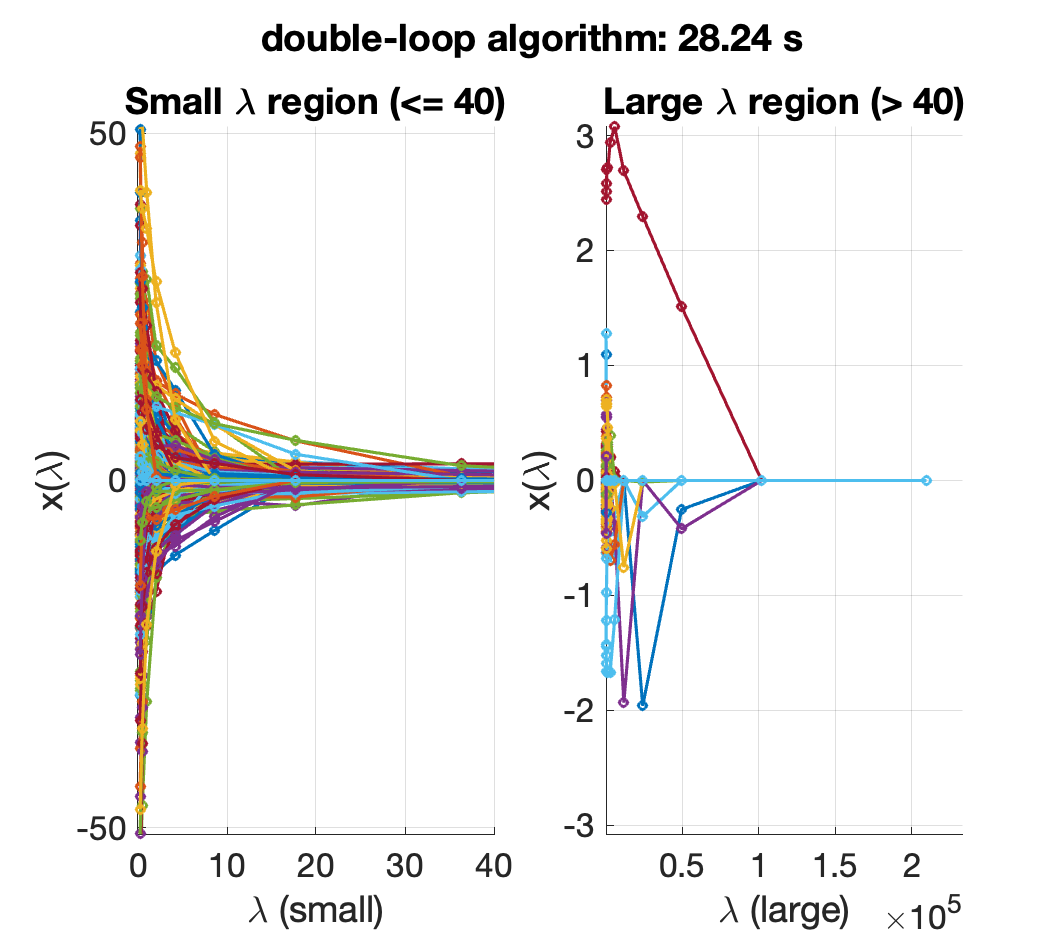}
	\end{subfigure}}
	\caption{Comparison of path generation between our algorithm and \texttt{SparseReg}'s path following algorithm on microbiome compositional regression $(m=932,n=1000)$.}
	\label{fig: micro_path_ppa_1000}
\end{figure}

We first compare our algorithm with the path-following solver, which was reported to outperform the other two methods in \cite{gaines2018algorithms}. This solver automatically generates a sequence of $\lambda$ values for \eqref{eq: micro_ls} and computes solutions by tracking  solution path events. In contrast, our algorithm constructs the solution path by using an explicitly specified $\lambda$ sequence. In our experiments, we set $\lambda = \varrho \|A^\top b\|$ with $\varrho$ decreasing from $0.9$ to $10^{-6}$ over 20 points equally spaced on the $\log_{10}$ scale, roughly matching the range of $\lambda$ values generated by the path-following solver. As is standard in path generation, we initialize each problem using the previous solution at the larger $\lambda$. Notably, our algorithm allows flexible user-defined $\lambda$ sequences, whereas \texttt{SparseReg}'s path-following solver does not. Preliminary experiments show that the path-following solver in \texttt{SparseReg} scales poorly on large instances, so we restrict the experiments to datasets with 1,000 and 3,000 OTUs; see Figures~\ref{fig: micro_path_ppa_1000} and~\ref{fig: micro_path_ppa_3000}. For both methods, we plot the coefficient trajectories along the paths. We split the display into a small $\lambda$ regime (with dense solutions) and a large $\lambda$ regime (with sparse solutions), each with its own axis scaling to keep both regimes clearly illustrated.
As shown in the two figures, our algorithm achieves a nearly identical solution path to that of \texttt{SparseReg}'s path-following algorithm, while requiring significantly less computation time.

\begin{figure}[H]
	{\begin{subfigure}[b]{0.5\textwidth}
		\centering
		\includegraphics[width=\textwidth]{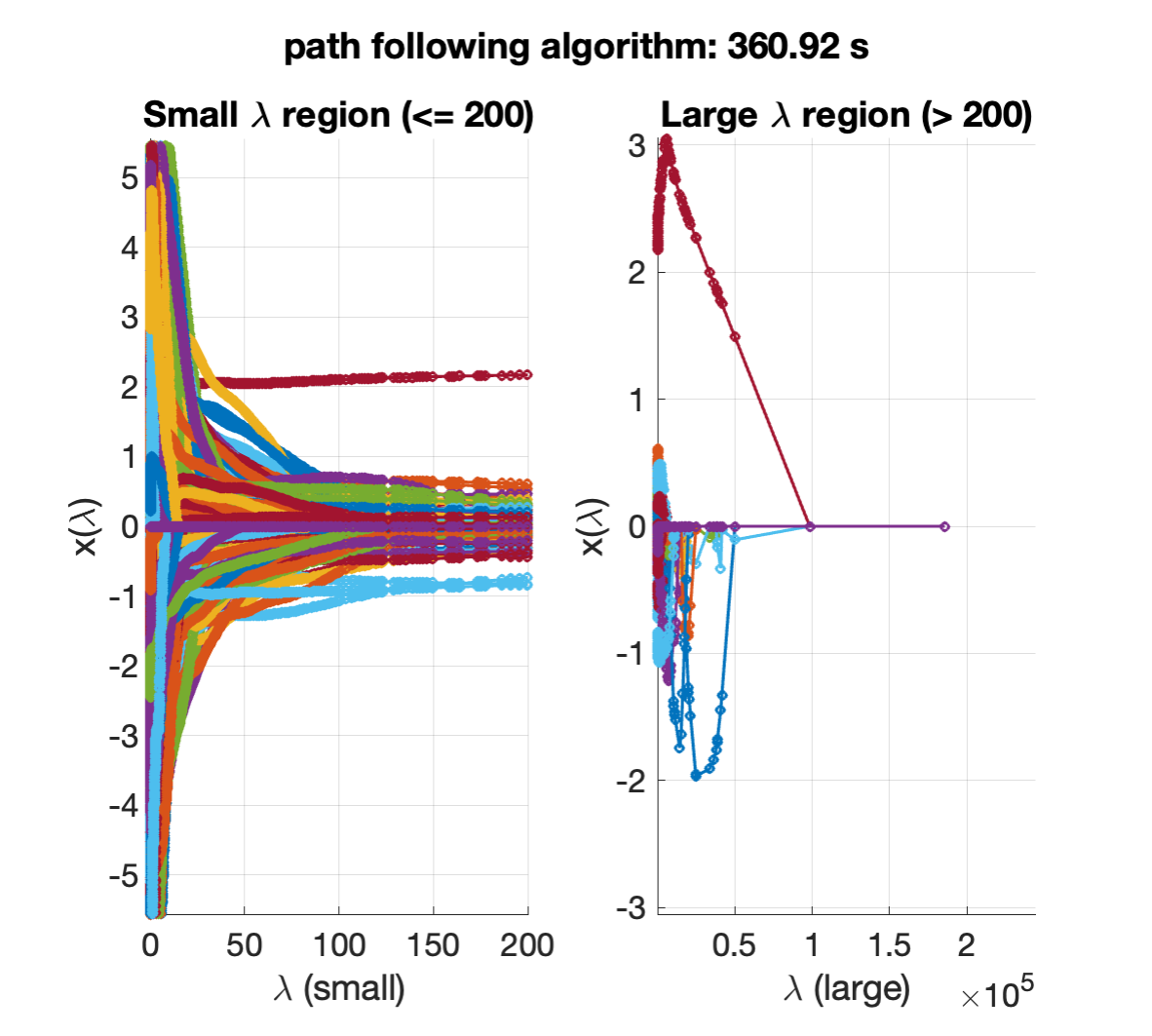}
	\end{subfigure}
	\begin{subfigure}[b]{0.5\textwidth}
		\centering
		\includegraphics[width=\textwidth]{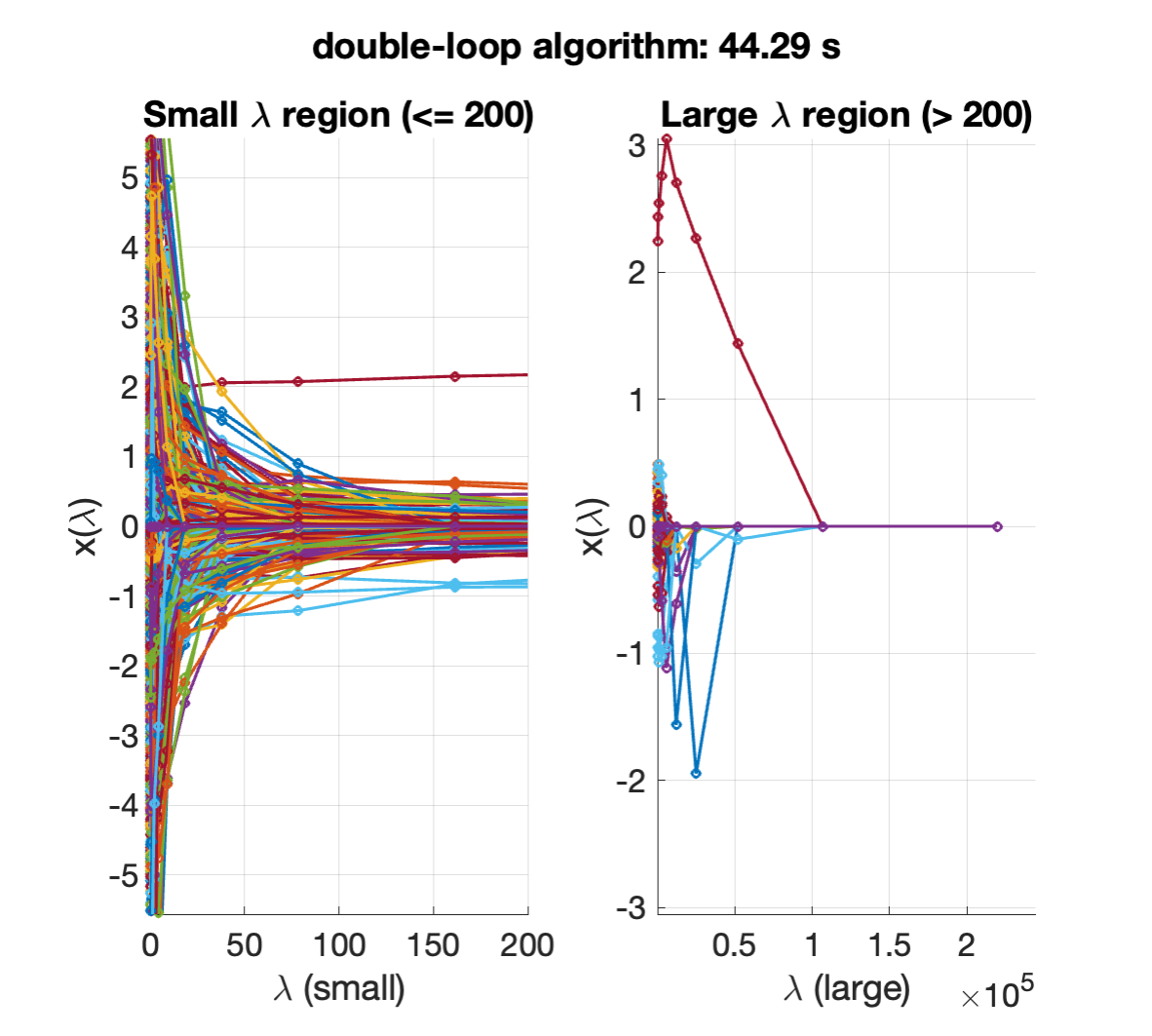}
	\end{subfigure}}
	\caption{Comparison of  path generation between our algorithm and \texttt{SparseReg}'s path following algorithm on microbiome compositional regression $(m=932,n=3000)$.}
	\label{fig: micro_path_ppa_3000}
\end{figure}

Next, we compare our algorithm with the ADMM solver in \texttt{SparseReg}. Similar to our approach, ADMM requires a pre-specified sequence of $\lambda$ values to generate solution paths for the microbiome compositional regression problem \eqref{eq: micro_ls}. Based on our experiments, \texttt{SparseReg}'s ADMM solver fails to solve problem \eqref{eq: micro_ls} for small $\lambda$ values and does not scale to large problem sizes. To better visualize and compare the performance of both methods, we restrict the range of $\lambda$ to $\lambda = \varrho \|A^\top b\|$, where $\varrho$ decreases from $0.9$ down to $10^{-4}$ for 1,000 OTUs case and down to $10^{-3}$ for 3,000 OTUs case, using 10 logarithmically spaced grid points. The runtime comparison is shown in Figure~\ref{fig: micro_admm_ppa}. In both cases, our algorithm runs significantly faster than ADMM. Specifically, on each dataset, it computes the full solution path within 10 seconds, whereas ADMM takes at least 20 seconds to solve a single subproblem.

\begin{figure}[H]
	{\begin{subfigure}[b]{0.5\textwidth}
			\centering
			\includegraphics[width=\textwidth]{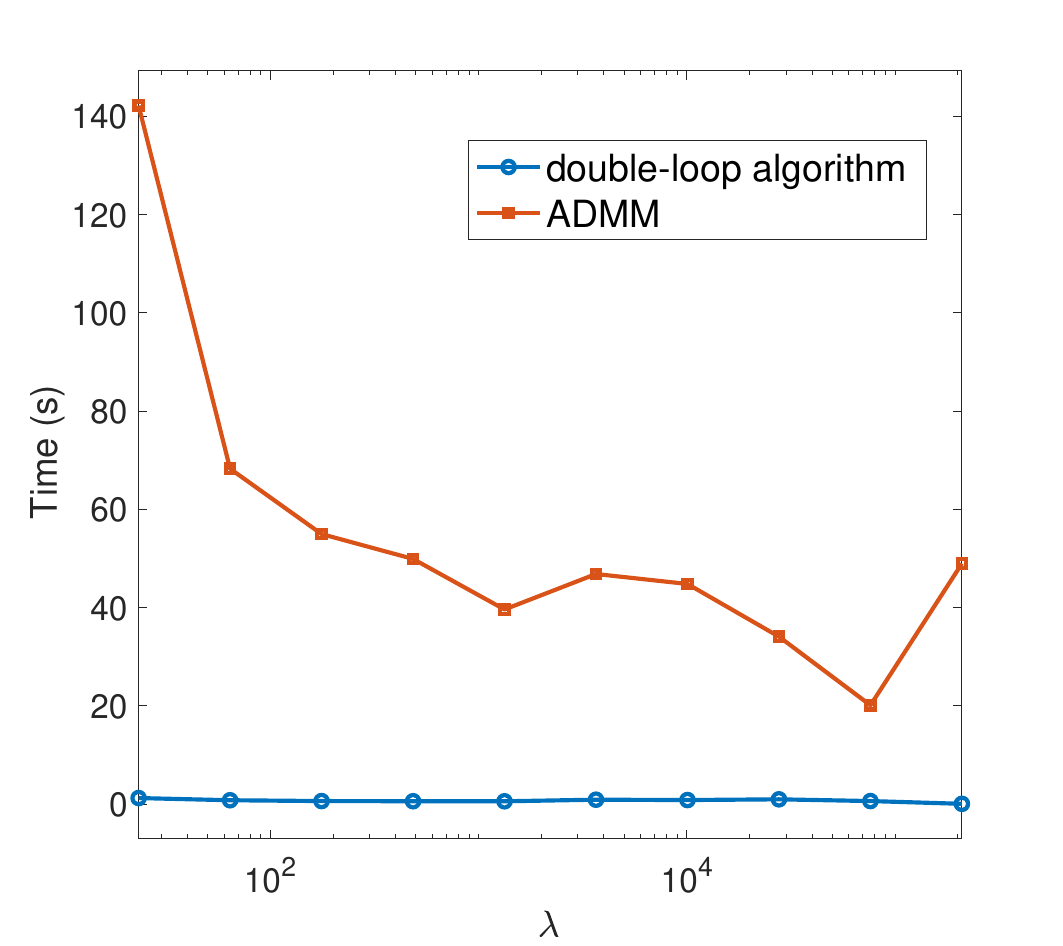}
			\caption{ $m=932,n=1000$.}
		\end{subfigure}
		\begin{subfigure}[b]{0.5\textwidth}
			\centering
			\includegraphics[width=\textwidth]{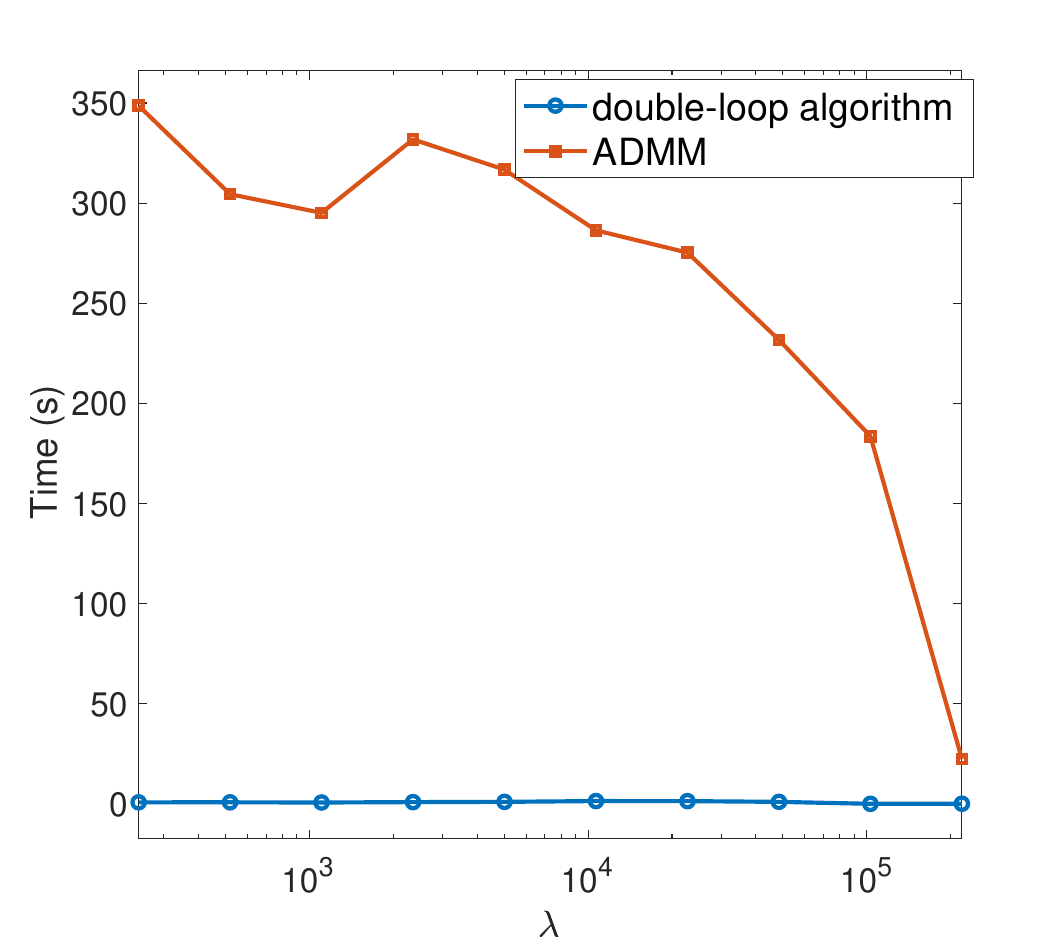}
			\caption{ $m=932,n=3000$.}
	\end{subfigure}}
	\caption{Runtime comparison of path generation between our algorithm and \texttt{SparseReg}'s ADMM on microbiome compositional regression with varying sizes. }
	\label{fig: micro_admm_ppa}
\end{figure}

\begin{figure}[H]
	\centering
	{\includegraphics[width=0.5\textwidth]{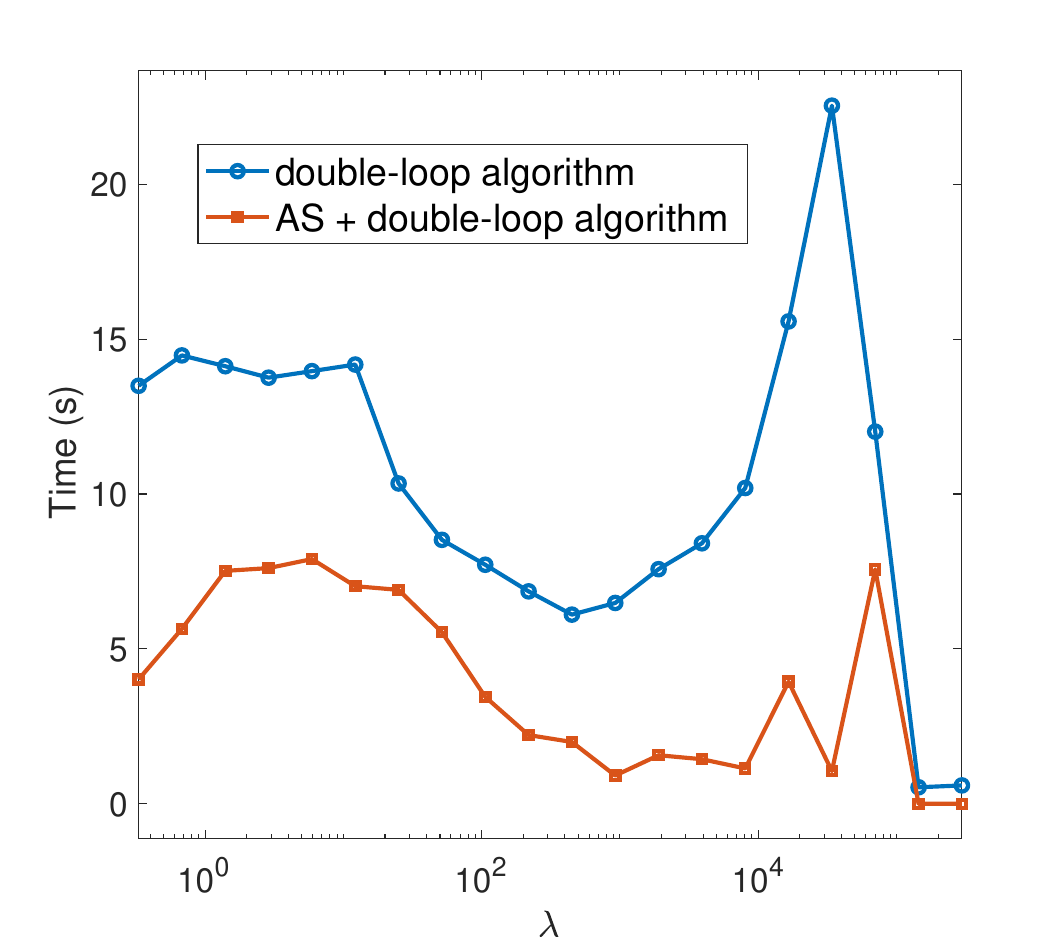}}
	\caption{Runtime of our algorithm with or without AS, for path generation on microbiome compositional regression $(m=932,n=209,356)$.}\label{fig: micro_AS_PPA}
\end{figure}

It should be noted that, due to the limitations of the path-following and ADMM solvers in \texttt{SparseReg}, the above experiments are restricted to relatively small-scale problems. In contrast, our proposed double-loop algorithm is capable of handling large-scale datasets efficiently. We evaluate the performance of our algorithm on the full dataset, which consists of 932 samples and 209,356 OTUs. We set $\lambda = \varrho \|A^\top b\|$, where $\varrho$ ranges from $0.9$ down to $10^{-6}$, using 20 grid points equally spaced on the $\log_{10}$ scale. To further enhance efficiency, our double-loop algorithm can be combined with the adaptive sieving (AS) strategy~\cite{yuan2025adaptive}, a powerful dimension reduction technique for sparse optimization problems. Figure~\ref{fig: micro_AS_PPA} shows the performance of our algorithm on the full set, both with and without AS. As illustrated, our algorithm successfully solves the full-scale problem within a reasonable time, and the AS strategy further accelerates computation. These results demonstrate that our algorithm not only scales to large datasets, but also benefits from the AS strategy for high computational efficiency.

\subsubsection{Classification analysis}
Beyond regression problems, we further examine the performance of our algorithm for solving the microbiome compositional classification problem \eqref{eq: micro_logistic}. To provide a meaningful benchmark, we compare our algorithm with \texttt{ECLasso} (\url{https://github.com/lamttran/ECLasso}), a recently proposed state-of-the-art R package that is specifically designed to fit logistic regression models with a lasso penalty while incorporating linear constraints, via candidate subsets identified from the unconstrained lasso \cite{tran2024fast}.

Our preliminary tests indicate that \texttt{ECLasso} exhibits limited efficiency on datasets with a relatively large number of samples or features for this problem. Consequently, we restrict the comparison between two methods to a small dataset consisting of 50 samples and 60 OTUs. Figure \ref{fig: micro_logistic} summarizes the results. Both methods are evaluated using 20 $\lambda$ values sampled on a logarithmic scale between 5 and 0.15. As shown in the figure, our algorithm significantly outperforms \texttt{ECLasso} in computational efficiency while providing comparable solutions along the path. In particular, our algorithm computes the entire solution path in just around one second, whereas \texttt{ECLasso} requires more than 120 seconds. Although the two methods are implemented in different environments, with our algorithm in {\sc Matlab} and \texttt{ECLasso} in R, the substantial performance gap underscores the practical advantage of our approach. 

\begin{figure}[H]
	{\begin{subfigure}[b]{0.5\textwidth}
		\centering
		\includegraphics[width=\textwidth]{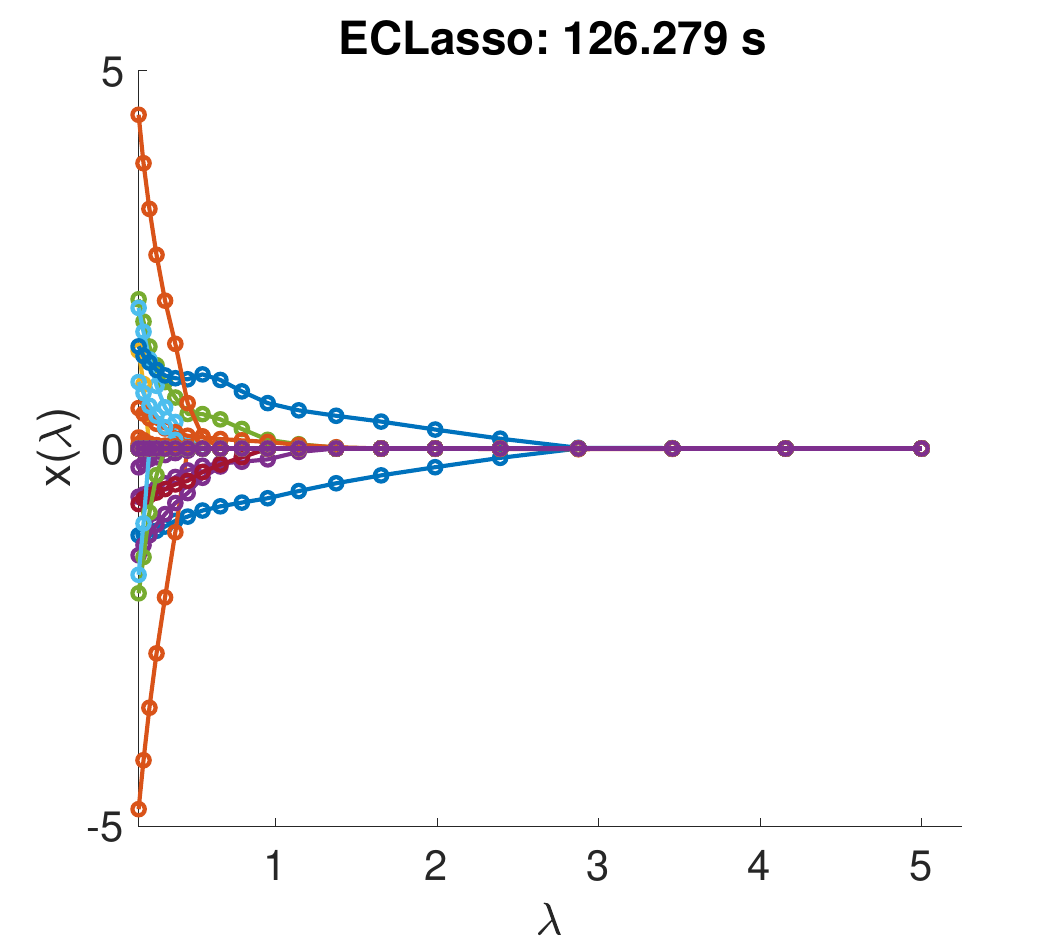}
	\end{subfigure}
	\begin{subfigure}[b]{0.5\textwidth}
		\centering
		\includegraphics[width=\textwidth]{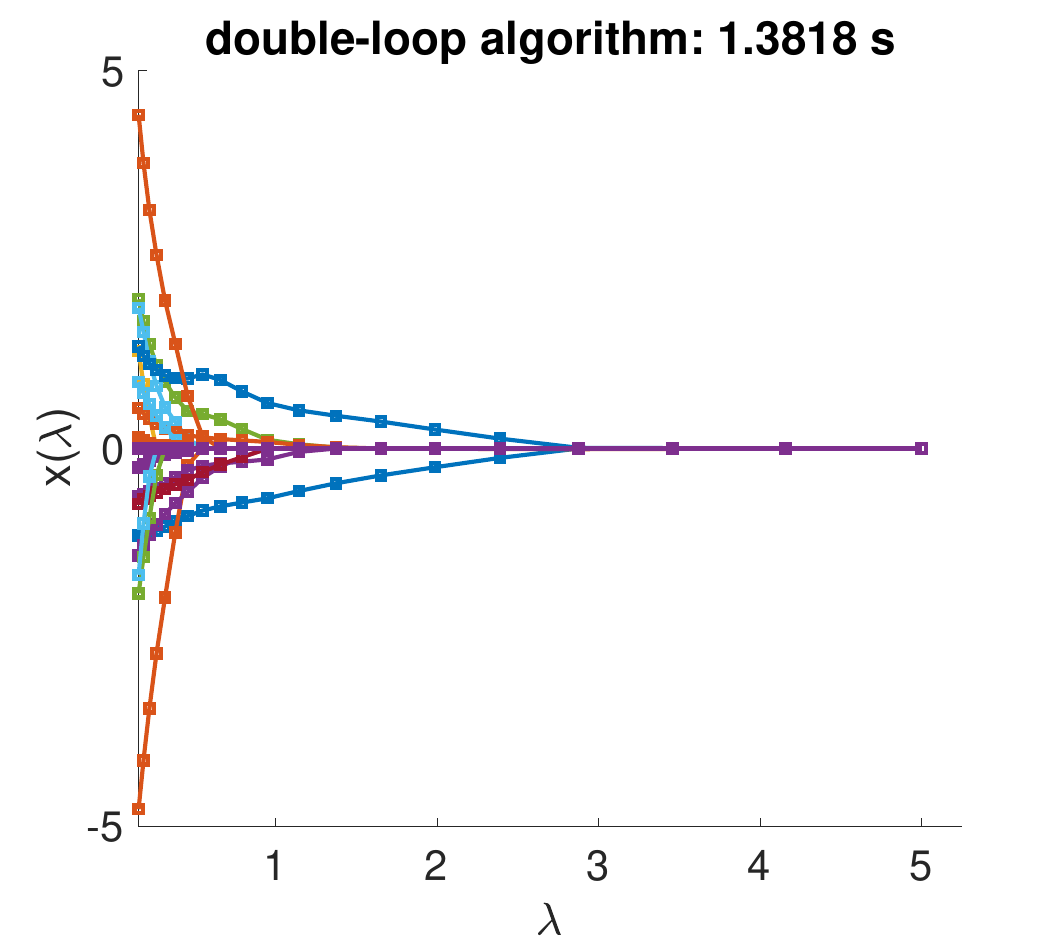}
	\end{subfigure}}
    \caption{Comparison of path generation between our algorithm and \texttt{ECLasso} on microbiome compositional classification $(m=50,n=60)$.}
    \label{fig: micro_logistic}
\end{figure}

To further demonstrate the scalability of our algorithm for solving problem \eqref{eq: micro_logistic} beyond the small-scale comparison with \texttt{ECLasso}, we evaluate its performance on larger datasets, including the full dataset with 932 samples and 209,356 OTUs. Table~\ref{tab:logistic_ppa_times} presents results under both the standard setting and its AS-enhanced variant. For each dataset, we generate a solution path by solving problem \eqref{eq: micro_logistic} at a sequence of 10 $\lambda$ values, where each $\lambda = \varrho \|A^\top b\|$ and the $\varrho$ values are logarithmically spaced from 0.5 down to $10^{-5}$. The results demonstrate that our algorithm scales effectively with both sample size and dimensionality. Even for the largest problem, which involves over 190 million parameters in the feature matrix, our method computes the full path in 19 minutes, and just 6 minutes when combined with AS.

\begin{table}[H]
	\caption{Performance of our algorithm for path generation on microbiome compositional classification across varying problem sizes, reported under the standard and AS-enhanced settings. Here, ``nnz" denotes the number of nonzeros of the solution at the smallest $\lambda$. Time is shown in (minutes:seconds).} \label{tab:logistic_ppa_times}
	\centering
	{\begin{tabular}{|c|c|c|cc|}
		\hline
		$m$ & $n$ & nnz & Standard Time & AS-Enhanced Time  \\
		\hline
		200 & 50000  & 128 & 00:54  & 00:23 \\
		200 & 100000 & 126 & 01:15  & 00:25 \\
		200 & 209356 & 134 & 02:16  & 00:35 \\
		500 & 50000  & 288 & 02:04  & 01:07 \\
		500 & 100000 & 309 & 03:05  & 01:17 \\
		500 & 209356 & 297 & 06:22  & 01:34 \\
		932 & 50000  & 565 & 07:03  & 05:13 \\
		932 & 100000 & 554 & 19:07  & 05:46 \\
		932 & 209356 & 573 & 18:57  & 05:49 \\
		\hline
	\end{tabular}}
\end{table}

\subsection{Sparse subspace clustering}
In this subsection, we evaluate the performance of our double-loop algorithm on sparse subspace clustering. As noted in the introduction, the original matrix formulation \eqref{eq: sub_cluster_matrix} can be decomposed into $n$ vectorized problems of the form \eqref{eq: sub_cluster_vector}. These can be solved individually or handled jointly by adapting our algorithm to the matrix form. We adopt the latter approach to avoid for-loops and improve implementation efficiency. Existing sparse subspace clustering methods often struggle with large sample sizes due to the need to solve the $n \times n$ optimization problem \eqref{eq: sub_cluster_matrix} and perform spectral clustering on large affinity matrices \cite{peng2013scalable,traganitis2017sketched,pourkamali2019large,abdolali2019scalable}. To address this issue, techniques such as random sketching \cite{traganitis2017sketched}, anchor point selection via hierarchical clustering \cite{abdolali2019scalable}, and landmark-based methods \cite{peng2013scalable,pourkamali2019large} have been proposed. A detailed discussion on these approaches is beyond the scope of this work. In our experiments, we follow the landmark-based approach \cite{peng2013scalable,pourkamali2019large}, solving \eqref{eq: sub_cluster_matrix} over a set of representative landmarks to effectively reduce the problem size. 

We conduct experiments on three real-world datasets  (\url{https://github.com/XLearning-SCU/2013-CVPR-SSSC/tree/master}): the Covertype dataset (581,012 samples, 54 features), the Pendigits dataset (10,992 samples, 16 features), and the Pokerhand dataset (1,000,000 samples, 10 features). We compare our double-loop algorithm against two existing sparse subspace clustering methods for solving \eqref{eq: sub_cluster_matrix}: an ADMM-based solver \cite{pourkamali2020efficient} and a proximal gradient method with Nesterov acceleration from the TFOCS package \cite{becker2011templates,pourkamali2020efficient}. Both baselines are publicly available with core routines implemented in {\sc Matlab} (\url{https://github.com/stephenbeckr/SSC}). 

Figure~\ref{fig:comparison_pokerhand} compares the three methods on Pokerhand dataset, using landmark sizes of 300 and 500, with $\lambda = 10^{-4}$ as recommended in \cite{peng2013scalable}. Both our double-loop algorithm and TFOCS are theoretically guaranteed to maintain feasibility, and in practice exhibit near-feasibility throughout the iterations. In contrast, ADMM begins with significant infeasibility, which diminishes slowly over iterations but remains non-negligible. To ensure a fair comparison, we report not only the objective values against computational time but also the feasibility of ADMM. As can be seen in the figure, our algorithm consistently achieves lower objective values in less time across both landmark sizes. While ADMM produces comparable objective values in the 500-landmark case, it suffers from poor constraint satisfaction, failing to meet $X^\top e = e$.

\begin{figure}[H]
	{\begin{subfigure}[b]{0.5\textwidth}
		\centering
		\includegraphics[width=\textwidth]{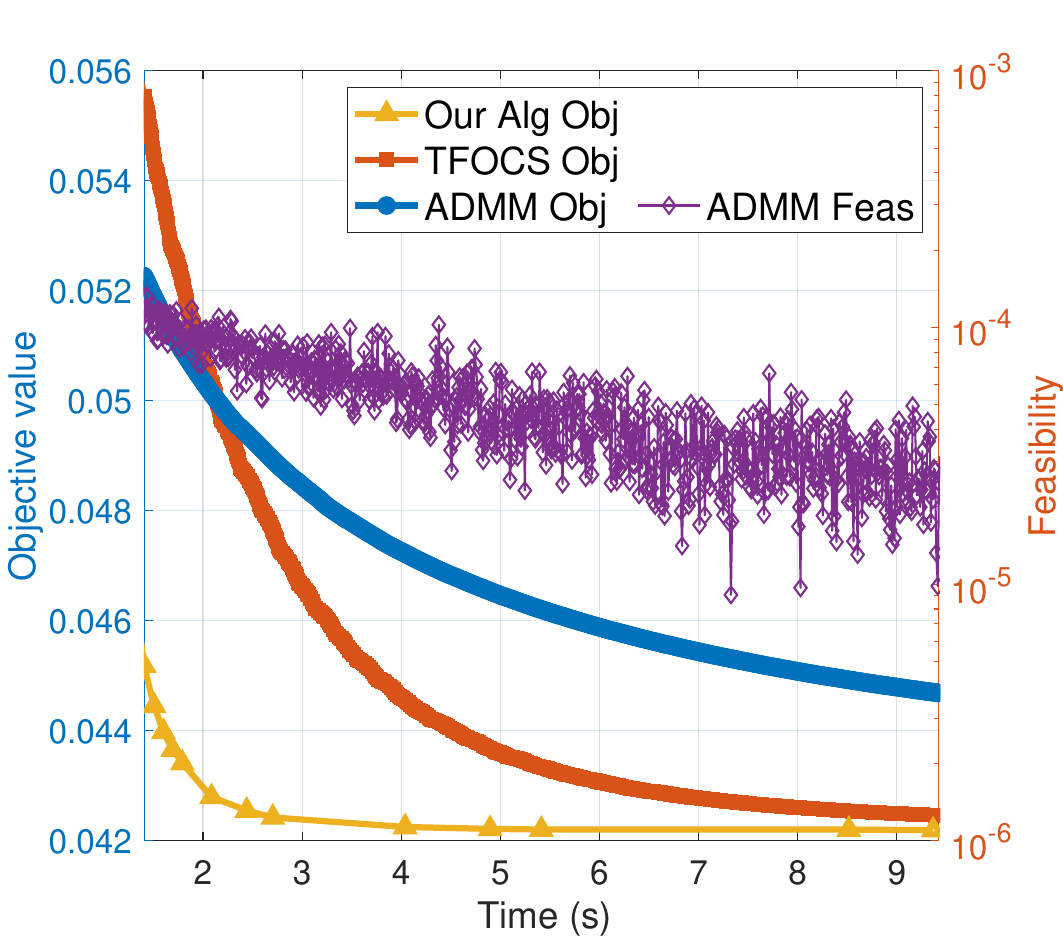}
	\end{subfigure}
	\begin{subfigure}[b]{0.5\textwidth}
		\centering
		\includegraphics[width=\textwidth]{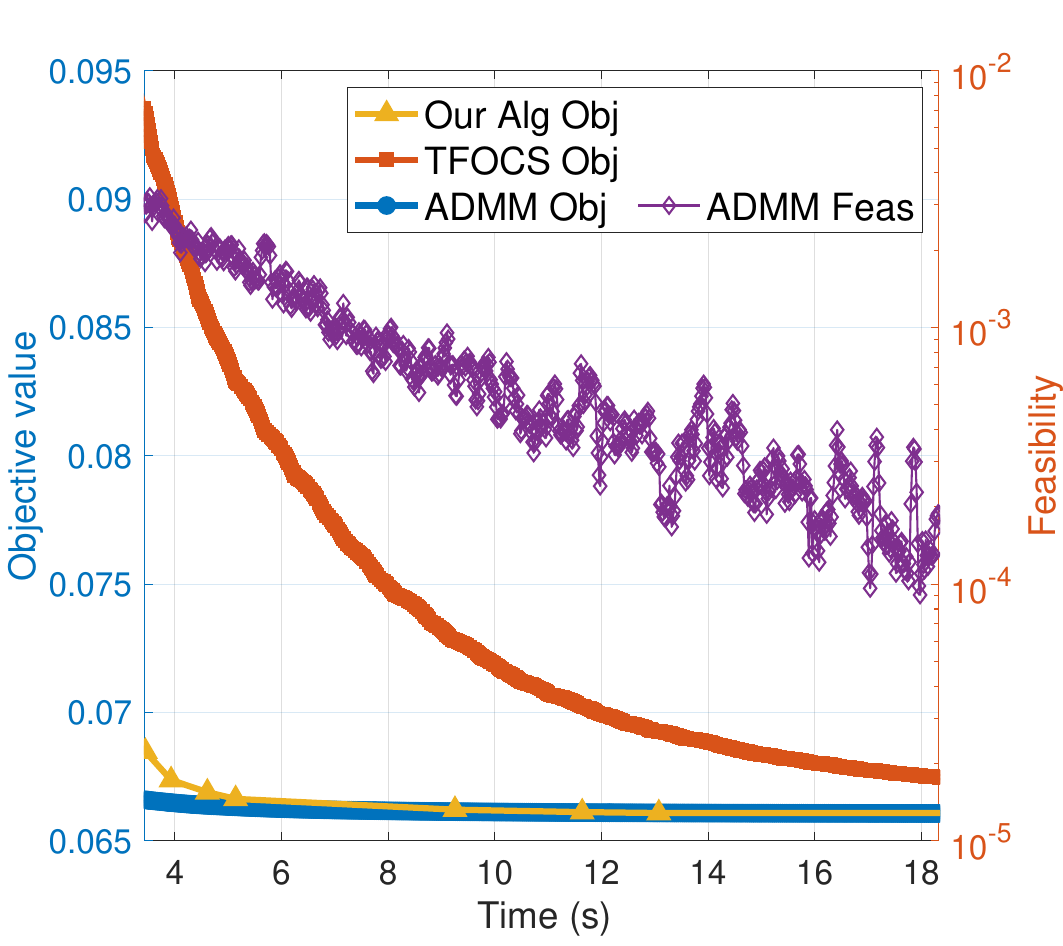}
	\end{subfigure}}
\caption{Comparison among our double-loop algorithm, TFOCS, and ADMM for sparse subspace clustering on Pokerhand dataset, with landmark size $300$ (Left) and $500$ (Right). Objective values over time are shown for all methods; constraint violation is reported only for ADMM, as the others maintain near-feasibility throughout.}
\label{fig:comparison_pokerhand}
\end{figure}

\begin{table}[H]
	\caption{Comparison of (a) our double-loop algorithm, (b) ADMM, and (c) TFOCS for sparse subspace clustering on Covtype dataset, using varying landmark sizes and $\lambda$ values. ``Normalized Obj." denotes the relative objective difference, computed as $(\mbox{objective-ours})/\mbox{ours}$. ``Feasibility" measures constraint violation as $\|X^\top e - e\|_F$.}
	\label{table: covtype}
	{\begin{tabular}{|c|c|c @{\ $|$\ } c @{\ $|$\ } c|c @{\ $|$\ } c @{\ $|$\ } c|c @{\ $|$\ } c @{\ $|$\ } c|}
			\hline
			Prob. & & \multicolumn{3}{c|}{Time (mm:ss)} & \multicolumn{3}{c|}{Normalized Obj.} & \multicolumn{3}{c|}{Feasibility} \\
			\hline
			$(m,n)$ & $\lambda$ & a & b & c & a & b & c & a & b & c \\
			\hline
			\multirow{3}{*}{(54, 200)} 
			& 1e-3 & 00:06 & 02:00 & 00:59 & 0 & 2.01e-5 & 3.06e-6 & 7.37e-12 & 7.28e-7 & 3.38e-14 \\
			& 1e-4 & 00:08 & 02:02 & 01:01 & 0 & 1.24e-2 & 4.40e-5 & 4.94e-12 & 1.37e-6 & 2.61e-14 \\
			& 1e-5 & 00:14 & 02:33 & 01:15 & 0 & 1.09e-1 & 5.86e-4 & 5.74e-13 & 8.53e-7 & 2.42e-14 \\
			\hline
			\multirow{3}{*}{(54, 400)} 
			& 1e-3 & 00:42 & 04:15 & 03:35 & 0 & 7.42e-4 & 3.31e-6 & 1.14e-11 & 4.88e-6 & 3.79e-14 \\
			& 1e-4 & 00:41 & 03:28 & 03:09 &0  & 2.08e-5 & 5.61e-5 & 5.43e-12 & 7.67e-7 & 3.52e-14 \\
			& 1e-5 & 00:48 & 04:30 & 03:34 & 0 & 4.25e-3 & 7.82e-4 & 1.16e-12 & 4.06e-6 & 4.55e-14 \\
			\hline
			\multirow{3}{*}{(54, 600)} 
			& 1e-3 & 01:58 & 09:20 & 07:41 & 0 & 7.19e-5 & 3.99e-6 & 4.66e-12 & 8.01e-6 & 4.46e-14 \\
			& 1e-4 & 02:14 & 09:33 & 07:34 & 0& 3.83e-5 & 6.67e-5 & 1.07e-11 & 3.00e-5 & 6.35e-14 \\
			& 1e-5 & 03:27 & 09:41 & 06:57 & 0 & 1.68e-5 & 1.03e-3 & 2.15e-12 & 1.04e-5 & 5.09e-14 \\
			\hline
			\multirow{3}{*}{(54, 800)} 
			& 1e-3 & 03:57 & 14:28 & 13:32 & 0 & 1.13e-4 & 4.28e-6 & 3.96e-12 & 1.41e-5 & 5.68e-14 \\
			& 1e-4 & 03:55 & 15:07 & 13:33 & 0 & 4.90e-5 & 7.32e-5 & 7.20e-12 & 2.50e-5 & 5.74e-14 \\
			& 1e-5 & 05:03 &16:28 & 12:39 & 0 & 1.83e-5 & 1.24e-3 & 2.20e-12 & 2.22e-5 & 4.81e-14 \\
			\hline
	\end{tabular}}
\end{table}

We further compare the performance of our double-loop algorithm, ADMM, and TFOCS on all three datasets for solving \eqref{eq: sub_cluster_matrix} under various landmark sizes and $\lambda$ values. The results are summarized in Tables~\ref{table: covtype}--\ref{table: pokerhand}. In the tables, $m$ denotes the feature dimension of the data, and $n$ refers to the number of selected landmarks used in the clustering formulation \eqref{eq: sub_cluster_matrix}; the normalized objective reflects the relative gap between each method’s objective value and that of our algorithm, while feasibility measures constraint violation as $\|X^\top e - e\|_F$. As shown, our algorithm consistently achieves the lowest objective values in the shortest time while maintaining acceptable feasibility, highlighting its efficiency and robustness compared to existing methods.

\begin{table}[H]
	\caption{Comparison of (a) our double-loop algorithm, (b) ADMM, and (c) TFOCS for sparse subspace clustering on Pendigits dataset.}
	\label{table: pendigits}
	{\begin{tabular}{|c|c|c @{\ $|$\ } c @{\ $|$\ } c|c @{\ $|$\ } c @{\ $|$\ } c|c @{\ $|$\ } c @{\ $|$\ } c|}
		\hline
		Prob.  & & \multicolumn{3}{c|}{Time (mm:ss)} & \multicolumn{3}{c|}{Normalized Obj.} & \multicolumn{3}{c|}{Feasibility} \\
		\hline
		$(m,n)$ & $\lambda$ & a & b & c & a & b & c & a & b & c \\
		\hline
		\multirow{3}{*}{(16, 200)} 
		& 1e-3 & 00:07 & 01:38 & 00:56 & 0 & 4.56e-6 & 5.30e-6 & 3.64e-12 & 2.13e-7 & 3.68e-14 \\
		& 1e-4 & 00:09 & 01:37 & 00:57 & 0 & 8.48e-7 & 6.69e-5 & 2.80e-12 & 2.88e-8 & 3.57e-14 \\
		& 1e-5 & 00:12 & 01:40 & 00:58 & 0 & 5.87e-2 & 7.54e-4 & 1.26e-12 & 8.10e-7 & 3.49e-14 \\
		\hline
		\multirow{3}{*}{(16, 400)} 
		& 1e-3 & 00:26 & 02:52 & 03:25 & 0 & 1.52e-5 & 6.86e-6 & 6.05e-12 & 3.94e-7 & 4.98e-14 \\
		& 1e-4 & 00:43 & 03:02 & 03:21 & 0 & 6.18e-6 & 8.41e-5 & 8.66e-12 & 6.06e-7 & 4.12e-14 \\
		& 1e-5 & 01:02 & 03:05 & 03:21 & 0 & 7.00e-2 & 1.11e-3 & 8.75e-13 & 1.27e-6 & 3.04e-14 \\
		\hline
		\multirow{3}{*}{(16, 600)} 
		& 1e-3 & 01:36 & 05:33 & 06:29 & 0 & 7.29e-5 & 7.67e-6 & 6.05e-12 & 2.99e-6 & 4.54e-14 \\
		& 1e-4 & 02:18 & 06:11 & 07:06 & 0 & 6.62e-6 & 9.41e-5 & 8.86e-12 & 1.63e-6 & 6.65e-14 \\
		& 1e-5 & 03:30 & 05:05 & 06:11 & 0 & -8.46e-10 & 1.48e-3 & 1.66e-12 & 1.02e-5 & 3.71e-14 \\
		\hline
		\multirow{3}{*}{(16, 800)} 
		& 1e-3 & 02:09 & 07:05 & 10:29 & 0 & 6.54e-6 & 8.02e-6 & 1.32e-11 & 5.91e-6 & 5.30e-14 \\
		& 1e-4 & 04:23 & 09:51 & 10:59 & 0 & 6.67e-6 & 1.04e-4 & 3.79e-12 & 1.23e-5 & 5.00e-14 \\
		& 1e-5 & 06:37 & 08:55 & 10:59 & 0 & 4.45e-6 & 1.66e-3 & 1.23e-12 & 2.28e-5 & 5.18e-14 \\
		\hline
	\end{tabular}}
\end{table}

\begin{table}[H]
	\caption{Comparison of (a) our double-loop algorithm, (b) ADMM, and (c) TFOCS for sparse subspace clustering on Pokerhand dataset. }
	\label{table: pokerhand}
	{\begin{tabular}{|c|c|c @{\ $|$\ } c @{\ $|$\ } c|c @{\ $|$\ } c @{\ $|$\ } c|c @{\ $|$\ } c @{\ $|$\ } c|}
		\hline
		Prob. & & \multicolumn{3}{c|}{Time (mm:ss)} & \multicolumn{3}{c|}{Normalized Obj.} & \multicolumn{3}{c|}{Feasibility} \\
		\hline
		$(m,n)$ & $\lambda$ & a & b & c & a & b & c & a & b & c \\
		\hline
		\multirow{3}{*}{(10, 200)} 
		& 1e-3 & 00:04 & 01:47 & 00:53 & 0 & 5.73e-5 & 4.47e-6 & 9.01e-12 & 9.47e-8 & 2.73e-14 \\
		& 1e-4 & 00:07 & 01:43 & 00:55 & 0 & 2.41e-2 & 4.59e-5 & 4.90e-12 & 2.18e-6 & 2.54e-14 \\
		& 1e-5 & 00:08 & 01:40 & 00:54 & 0 & 1.02e-1 & 4.72e-4 & 5.16e-13 & 6.27e-8 & 2.66e-14 \\
		\hline
		\multirow{3}{*}{(10, 400)} 
		& 1e-3 & 00:15 & 02:20 & 02:42 & 0 & 1.02e-5 & 5.25e-6 & 8.43e-12 & 1.93e-6 & 3.11e-14 \\
		& 1e-4 & 00:26 & 02:22 & 02:44 & 0 & 1.33e-2 & 5.47e-5 & 1.93e-12 & 4.14e-6 & 3.31e-14 \\
		& 1e-5 & 00:32 & 02:25 & 02:47 & 0 & 1.87e-3 & 6.55e-4 & 1.07e-12 & 1.80e-5 & 4.66e-14 \\
		\hline
		\multirow{3}{*}{(10, 600)} 
		& 1e-3 & 00:50 & 04:50 & 05:52 & 0 & 1.08e-5 & 5.20e-6 & 8.37e-12 & 1.93e-6 & 4.87e-14 \\
		& 1e-4 & 01:33 & 04:52 & 05:54 & 0 & 1.26e-5 & 6.45e-5 & 9.80e-12 & 4.05e-6 & 3.79e-14 \\
		& 1e-5 & 02:07 & 04:03 & 05:56 & 0 & -8.38e-10 & 7.82e-4 & 1.99e-12 & 2.52e-5 & 6.57e-14 \\
		\hline
		\multirow{3}{*}{(10, 800)} 
		& 1e-3 & 01:38 & 07:58 & 11:35 & 0 & 1.25e-5 & 5.85e-6 & 8.36e-12 & 2.17e-6 & 4.48e-14 \\
		& 1e-4 & 02:44 & 08:49 & 11:16 & 0 & 1.41e-5 & 7.04e-5 & 8.26e-12 & 1.75e-5 & 4.30e-14 \\
		& 1e-5 & 04:14 & 15:31 & 15:13 & 0 & 8.83e-7 & 9.46e-4 & 1.82e-12 & 4.39e-5 & 8.40e-14 \\
		\hline
	\end{tabular}}
\end{table}

\section{Conclusion}\label{sec: conclusion}
This work offers a characterization of the B-subdifferential of the proximal operator associated with affine-constrained $\ell_1$ regularizers, which enables the design of efficient second-order methods for optimization problems involving such regularizers. These results provide new insights into the variational behavior of nonsmooth constrained regularizers, and lead to algorithms that outperform existing solvers in both efficiency and solution quality across real-world applications
including affine-constrained lasso problem for microbiome compositional
data analysis.


\begin{thebibliography}{10}
	
	\bibitem{abdolali2019scalable}
	{\sc M.~Abdolali, N.~Gillis, and M.~Rahmati}, {\em Scalable and robust sparse
		subspace clustering using randomized clustering and multilayer graphs},
	Signal Processing, 163 (2019), pp.~166--180.
	
	\bibitem{aitchison1984log}
	{\sc J.~Aitchison and J.~Bacon-Shone}, {\em Log contrast models for experiments
		with mixtures}, Biometrika,  (1984), pp.~323--330.
	
	\bibitem{becker2011templates}
	{\sc S.~R. Becker, E.~J. Cand{\`e}s, and M.~C. Grant}, {\em Templates for
		convex cone problems with applications to sparse signal recovery},
	Mathematical Programming Computation, 3 (2011), pp.~165--218.
	
	\bibitem{elhamifar2013sparse}
	{\sc E.~Elhamifar and R.~Vidal}, {\em Sparse subspace clustering: {A}lgorithm,
		theory, and applications}, IEEE Transactions on Pattern Analysis and Machine
	Intelligence, 35 (2013), pp.~2765--2781.
	
	\bibitem{facchinei2003finite}
	{\sc F.~Facchinei and J.-S. Pang}, {\em Finite-dimensional Variational
		Inequalities and Complementarity Problems}, Springer Science \& Business
	Media, 2003.
	
	\bibitem{gaines2018algorithms}
	{\sc B.~R. Gaines, J.~Kim, and H.~Zhou}, {\em Algorithms for fitting the
		constrained lasso}, Journal of Computational and Graphical Statistics, 27
	(2018), pp.~861--871.
	
	\bibitem{james2020penalized}
	{\sc G.~M. James, C.~Paulson, and P.~Rusmevichientong}, {\em Penalized and
		constrained optimization: An application to high-dimensional website
		advertising}, Journal of the American Statistical Association,  (2020).
	
	\bibitem{li2018highly}
	{\sc X.~Li, D.~Sun, and K.-C. Toh}, {\em A highly efficient semismooth {N}ewton
		augmented {L}agrangian method for solving {L}asso problems}, SIAM Journal on
	Optimization, 28 (2018), pp.~433--458.
	
	\bibitem{lin2024highly}
	{\sc M.~Lin, Y.~Yuan, D.~Sun, and K.-C. Toh}, {\em A highly efficient algorithm
		for solving exclusive lasso problems}, Optimization Methods and Software, 39
	(2024), pp.~489--518.
	
	\bibitem{lin2014variable}
	{\sc W.~Lin, P.~Shi, R.~Feng, and H.~Li}, {\em Variable selection in regression
		with compositional covariates}, Biometrika, 101 (2014), pp.~785--797.
	
	\bibitem{lu2019generalized}
	{\sc J.~Lu, P.~Shi, and H.~Li}, {\em Generalized linear models with linear
		constraints for microbiome compositional data}, Biometrics, 75 (2019),
	pp.~235--244.
	
	\bibitem{moreau1965proximite}
	{\sc J.-J. Moreau}, {\em Proximit{\'e} et dualit{\'e} dans un espace
		hilbertien}, Bulletin de la Soci{\'e}t{\'e} math{\'e}matique de France, 93
	(1965), pp.~273--299.
	
	\bibitem{peng2013scalable}
	{\sc X.~Peng, L.~Zhang, and Z.~Yi}, {\em Scalable sparse subspace clustering},
	in Proceedings of the IEEE Conference on Computer Vision and Pattern
	Recognition, 2013, pp.~430--437.
	
	\bibitem{pourkamali2019large}
	{\sc F.~Pourkamali-Anaraki}, {\em Large-scale sparse subspace clustering using
		landmarks}, in 2019 IEEE 29th International Workshop on Machine Learning for
	Signal Processing (MLSP), IEEE, 2019, pp.~1--6.
	
	\bibitem{pourkamali2020efficient}
	{\sc F.~Pourkamali-Anaraki, J.~Folberth, and S.~Becker}, {\em Efficient solvers
		for sparse subspace clustering}, Signal Processing, 172 (2020), p.~107548.
	
	\bibitem{rockafellar1976monotone}
	{\sc R.~T. Rockafellar}, {\em Monotone operators and the proximal point
		algorithm}, SIAM Journal on Control and Optimization, 14 (1976),
	pp.~877--898.
	
	\bibitem{rockafellar1997convex}
	\leavevmode\vrule height 2pt depth -1.6pt width 23pt, {\em Convex Analysis},
	vol.~28, Princeton University Press, 1997.
	
	\bibitem{rockafellar2009variational}
	{\sc R.~T. Rockafellar and R.~J.-B. Wets}, {\em Variational Analysis},
	vol.~317, Springer Science \& Business Media, 2009.
	
	\bibitem{shi2016regression}
	{\sc P.~Shi, A.~Zhang, and H.~Li}, {\em Regression analysis for microbiome
		compositional data}, The Annals of Applied Statistics, 10 (2016), pp.~1019 --
	1040.
	
	\bibitem{sun2008lowner}
	{\sc D.~Sun and J.~Sun}, {\em L{\"o}wner's operator and spectral functions in
		{E}uclidean {J}ordan algebras}, Mathematics of Operations Research, 33
	(2008), pp.~421--445.
	
	\bibitem{susin2020variable}
	{\sc A.~Susin, Y.~Wang, K.-A. L{\^e}~Cao, and M.~L. Calle}, {\em Variable
		selection in microbiome compositional data analysis}, NAR Genomics and
	Bioinformatics, 2 (2020), p.~lqaa029.
	
	\bibitem{traganitis2017sketched}
	{\sc P.~A. Traganitis and G.~B. Giannakis}, {\em Sketched subspace clustering},
	IEEE Transactions on Signal Processing, 66 (2017), pp.~1663--1675.
	
	\bibitem{tran2024fast}
	{\sc L.~Tran, G.~Li, L.~Luo, and H.~Jiang}, {\em A fast solution to the lasso
		problem with equality constraints}, Journal of Computational and Graphical
	Statistics, 33 (2024), pp.~804--813.
	
	\bibitem{vidal2009sparse}
	{\sc E.~E.~R. Vidal et~al.}, {\em Sparse subspace clustering}, in 2009 IEEE
	Conference on Computer Vision and Pattern Recognition (CVPR), vol.~6, 2009,
	pp.~2790--2797.
	
	\bibitem{yuan2025adaptive}
	{\sc Y.~Yuan, M.~Lin, D.~Sun, and K.-C. Toh}, {\em Adaptive sieving: {A}
		dimension reduction technique for sparse optimization problems}, Mathematical
	Programming Computation, 17 (2025), pp.~585--616.
	
	\bibitem{zhao2010newton}
	{\sc X.-Y. Zhao, D.~Sun, and K.-C. Toh}, {\em A {N}ewton-{CG} augmented
		{L}agrangian method for semidefinite programming}, SIAM Journal on
	Optimization, 20 (2010), pp.~1737--1765.
	
	\bibitem{zhou2013path}
	{\sc H.~Zhou and K.~Lange}, {\em A path algorithm for constrained estimation},
	Journal of Computational and Graphical Statistics, 22 (2013), pp.~261--283.
	
\end{thebibliography}
\end{document}